\documentclass[12pt,a4paper]{amsart}
\usepackage[numbers]{natbib}
\usepackage{pst-node}
\usepackage{tikz-cd} 

\usepackage{enumitem,kantlipsum}
\usepackage[english]{babel}

\usepackage[T1]{fontenc}
\usepackage{enumitem}
\usepackage[utf8]{inputenc}

\usepackage{ifxetex}
\usepackage{pb-diagram}
\usepackage{amsfonts}
\usepackage{amsmath, amsthm, amscd, amsfonts, xpatch}
\usepackage{caption}
\usepackage{graphicx,float}
\usepackage{tikz}

\usepackage{lineno,hyperref}
\usepackage{graphicx,float}
\usepackage[english]{babel}
\usepackage{bussproofs}
\usepackage{csquotes}
\usepackage{dirtytalk}
\usepackage{mathtools}
\usepackage{mathrsfs}
\usepackage{latexsym}
\usepackage{enumitem}
\usepackage{amsthm}
\usepackage{amssymb}
\usepackage{amsthm}

\DeclareMathAlphabet{\mathpzc}{OT1}{pzc}{m}{it}
\usepackage{afterpage}

\usepackage{mathrsfs}

\usepackage[capitalise]{cleveref}

\input xy
\xyoption{all}

\newcommand{\ORD}{{\rm ORD}}
\newtheorem{theorem}{Theorem}[section]
\newtheorem{lemma}[theorem]{Lemma}

\newtheorem{proposition}[theorem]{Proposition}
\newtheorem{corollary}[theorem]{Corollary}

\newtheorem{definition}[theorem]{Definition}

\newtheorem{notation}[theorem]{Notation}
\newtheorem{claim}[theorem]{Claim}
\numberwithin{equation}{section}
\newtheorem{convention}[theorem]{Convention}

\newcommand{\cof}{{\rm cof}}
\newcommand{\dom}{{\rm dom}}

\newcommand{\M}{\mathcal{M}}
\newcommand{\rest}{\restriction}

\newcommand\axiom{\mathrm}

\newcommand\ZFC{\axiom{ZFC}}

\newcommand{\forces}{\Vdash}

\theoremstyle{remark}
\newtheorem{remark}[theorem]{Remark}
\makeatletter
\let\qed@empty\openbox 
\def\@begintheorem#1#2[#3]{%
\deferred@thm@head{%
\the\thm@headfont\thm@indent
\@ifempty{#1}
{\let\thmname\@gobble}
{\let\thmname\@iden}%
\@ifempty{#2}
{\let\thmnumber\@gobble\global\let\qed@current\qed@empty}
{\let\thmnumber\@iden\xdef\qed@current{#2}}%
\@ifempty{#3}
{\let\thmnote\@gobble}
{\let\thmnote\@iden}%
\thm@swap\swappedhead
\thmhead{#1}{#2}{#3}%
\the\thm@headpunct\thmheadnl\hskip\thm@headsep
}\ignorespaces
}
\renewcommand{\qedsymbol}{%
\ifx\qed@thiscurrent\qed@empty
\qed@empty
\else
\fbox{\scriptsize\qed@thiscurrent}%
\fi
}
\renewcommand{\proofname}{%
Proof%
\ifx\qed@thiscurrent\qed@empty
\else
\ of \qed@thiscurrent
\fi
}
\xpretocmd{\proof}{\let\qed@thiscurrent\qed@current}{}{}
\newenvironment{proof*}[1]
{\def\qed@thiscurrent{\ref{#1}}\proof}
{\endproof}
\makeatother
\begin{document}


\title{On Indestructible Strongly Guessing Models}

\author{Rahman Mohammadpour}
\address{Mathematical Institute of the Polish Academy of Sciences (Gdansk branch), Antoniego Abrahama 18, 81-825, Sopot, Poland}
\email{rahmanmohammadpour@gmail.com}
\urladdr{https://sites.google.com/site/rahmanmohammadpour}

\author{Boban Veli\v{c}kovi\'{c}}
\address{Institut de Math\'ematiques de Jussieu - Paris Rive Gauche, Universit\'e Paris  Cit\'e 8 Place Aur\'elie Nemours, 75205 Paris Cedex 13, France}
\email{boban@math.univ-paris-diderot.fr}
\urladdr{http://www.logique.jussieu.fr/~ boban}


\begin{abstract} 
In  \cite{MV} we defined and proved the consistency of the principle ${\rm GM}^+(\omega_3,\omega_1)$ which implies that many consequences
of strong forcing axioms hold simultaneously at  $\omega_2$ and $\omega_3$. 
In this paper we formulate a strengthening of ${\rm GM}^+(\omega_3,\omega_1)$ that we call ${\rm SGM}^+(\omega_3,\omega_1)$.
We also prove, modulo the consistency of two supercompact cardinals, that ${\rm SGM}^+(\omega_3,\omega_1)$ is consistent with ZFC.
In addition to all the consequences of ${\rm GM}^+(\omega_3,\omega_1)$, the principle ${\rm SGM}^+(\omega_3,\omega_1)$,
together with some mild cardinal arithmetic assumptions that hold in our model,  implies that any forcing that adds a new subset of $\omega_2$ 
either adds a real or collapses some cardinal. This gives a partial answer to   a question of Abraham  \cite{AvrahamPhD} and extends a previous result of Todor\v{c}evi\'{c} \cite{Todorcevic82} in this direction. 

\end{abstract}

\maketitle

\section{Introduction}
One of the driving themes of research in set theory in recent years has been the search for higher forcing axioms. 
Since most applications of strong forcing axioms such as the Proper Forcing Axiom (PFA) and Martin's Maximum (MM) 
can be factored through some simple, but powerful combinatorial principles, it is  natural to look for higher cardinal versions of such principles. 
One such principle is ${\rm ISP}(\omega_2)$  which was formulated by  Viale  and Wei\ss~\cite{VW2011}.
They derived it from PFA and showed that it has a number of  structural consequences on the set-theoretic universe. 
For instance, it implies the tree property at $\omega_2$, and  the failure of $\square(\lambda)$, for regular $\lambda \geq \omega_2$.
Moreover, Krueger \cite{Krueger2020} showed that it implies the Singular Cardinal Hypothesis (SCH).
A strengthening of ${\rm ISP}(\omega_2)$ was studied by Cox and Krueger \cite{CK2017}.
They showed that this strengthening, which we refer to as ${\rm SGM}(\omega_2,\omega_1)$, is consistent with $2^{\aleph_0} >\aleph_2$. 
It is therefore a natural candidate for a generalization to higher cardinals. 
In \cite{MV} we introduced one such generalization,  ${\rm GM}^+(\omega_3,\omega_1)$, and showed that if there are two supercompact cardinals
then there is a generic extension $V[G]$ of the universe in which ${\rm GM}^+(\omega_3,\omega_1)$ holds. 
In addition to ${\rm ISP}(\omega_2)$, the principle ${\rm GM}^+(\omega_3,\omega_1)$ implies that the tree property also holds at $\omega_3$, and
that the approachability ideal $I[\omega_2]$ restricted to ordinals of cofinality $\omega_1$ is the non stationary ideal restricted to this set. 
The consistency of the latter was previously shown by Mitchell \cite{MI2009} starting from a model with a greatly Mahlo cardinal. 
In the current paper we further strengthen the principle ${\rm GM}^+(\omega_3,\omega_1)$ in the spirit of Cox and Krueger \cite{CK2017}.
We call this new principle ${\rm SGM}^+(\omega_3,\omega_1)$. We show its consistency  by a modification of the forcing notion from \cite{MV}, 
again starting from two supercompact cardinals. 
As a new application we show that ${\rm SGM}^+(\omega_3,\omega_1)$, together with some mild cardinal arithmetic assumptions that hold
in our model, implies that any forcing that adds a new subset of $\omega_2$  either adds a real or collapses some cardinal. 
This gives a partial answer to a question of Abraham  \cite{AvrahamPhD} and extends a previous result of Todor\v{c}evi\'{c} \cite{Todorcevic82} in this direction. 

The paper is organized as follows. In \S 2 we recall some preliminaries and background facts. In \S3 we state the main theorem 
and derive the application of ${\rm SGM}^+(\omega_3,\omega_1)$ to the problem of Abraham mentioned above. 
In \S4 we review the definitions on virtual models from \cite{MV}. 
In \S5 we describe the main forcing and prove the relevant facts about it. The pieces are then put together in \S 6 to derive the main theorem of this paper. 
We warn the reader that the current paper relies heavily on the concepts and results from \cite{MV}, hence having access to that paper is necessary
for the full understanding of our results.

\section{Preliminaries}

\subsection*{Special trees}
Let  $T=(T,\leq_T)$ be a tree of height $\omega_1$. Recall that $T$ is called \emph{special} if there is a function $f:T\to \omega$,
such that  for every $s<_T t$ in $T$, we have  $f(s)\neq f(t)$. 
If $T$ is a tree of height $\omega_1$, the standard forcing $\mathbb S(T)$ to generically  specialize $T$ 
consists of finite partial specializing functions, ordering by reverse inclusion. 
If $T$ has no cofinal branches then $\mathbb S(T)$ has the countable chain condition, see \cite{BaumMalRei}. 
We say that $T$ is \emph{weakly special} if there is a function $f:T \to \omega$
such that whenever $s,t,u\in T$ and $f(s)=f(t)=f(u)$, if $s<_T t$ and $s<_Tu$ then $t$ and $u$ are comparable. 
Baumgartner \cite[Theorem 7.3]{Baumgartnersurvey} proved that for a  tree $T$ of height $\omega_1$ 
which has no branches of length $\omega_1$, $T$ is special iff $T$ is weakly special. 
He also showed \cite[Theorem 7.5]{Baumgartnersurvey} that if every tree of height and size $\omega_1$ that has no branches
of length $\omega_1$ is special, then every tree of height and size $\omega_1$ that has at most $\omega_1$ branches is weakly special. 
We recall the following well-known fact, see \cite[Proposition 4.3]{CK2017} for a proof.

\begin{proposition}\label{sealing} 
If $V\subseteq W$ are models of set theory with $\omega_1^V=\omega_1^W$, $T\in V$ is a tree of  height $\omega_1$ which
is weakly special in $V$, then every  branch of $T$ in $W$ of length $\omega_1$ is in $V$. 
\end{proposition} 
\begin{proof}[\nopunct]
\end{proof}

\subsection*{Guessing models}	
Throughout this paper by a model $M$, we mean a set or a class  such that  $(M,\in)$ satisfies a sufficient fragment of ${\rm ZFC}$.  For a set or class $M$, we say that a set $x\subseteq M$ is {\em bounded in} $M$ if there is $y\in M$ such that $x\subseteq y$. Let us call a transitive model $R$  {\em powerful}  if it includes every set bounded in $R$. We say that  $x$ is \emph{guessed} in $M$ if there is $x^*\in M$ with $x^*\cap M=x\cap M$. Suppose that $\gamma$ is an $M$-regular cardinal. We say $x$ is \emph{$\gamma$-approximated} in $M$ if for  every $a\in M$ with $|a|^M<\gamma$, we have $a\cap x\in M$. We say that $M$ is \emph{$\gamma$-guessing in a model} $N$ if  for every $x\in N$ bounded in $M$, if $x$ is $\gamma$-approximated in $M$, then $x$ is guessed in $M$.
	
\begin{definition}[\cite{V2012}]\label{guessing-def} 
Let $\gamma$ be a regular cardinal. 
A set  $M$ is said to be $\gamma$-{\em guessing} if  $M$ is $\gamma$-guessing in $V$.
\end{definition}
	
Recall from  \cite{Hamkins2001} that  a pair $(M,N)$ of transitive models with $\gamma\in M\subseteq N$ has the 
$\gamma$-{\em approximation property}, if  $M$ is $\gamma$-guessing in $N$.

\begin{theorem}[Krueger \cite{Krueger2020}]\label{transitivity of approx}
Let $M_0\subseteq M_1\subseteq M_2$ be transitive models. If $(M_0,M_1)$ and $(M_1,M_2)$ have the $\omega_1$-approximation property, then so does $(M_0,M_2)$.
\end{theorem}
\begin{proof}[\nopunct]
\end{proof}

\begin{theorem}[Cox--Krueger  \cite{CK-Quotient}]\label{guessing vs approx}
Suppose that $R$ is a powerful model and  $M\prec R$.   Let also $\gamma$ be a regular cardinal in $M$ with $\gamma\subseteq M$. Then the following are equivalent.
\begin{enumerate}
\item $M$ is a $\gamma$-guessing model.
\item The pair $(\overline M,V)$ has the $\gamma$-approximation property.
\end{enumerate} 
\end{theorem}
\begin{proof}[\nopunct]
\end{proof}
	
\subsection*{Indestructible and special guessing models}	
The notion of an indestructible guessing model was introduced by Cox and Krueger  \cite{CK2017}. 
\begin{definition}
An $\omega_1$-guessing model $M$ of size $\omega_1$ is \emph{indestructible} if $M$ remains $\omega_1$-guessing in any outer 
universe with the same $\omega_1$.
\end{definition}
Let us first recall that a set $X$ of size $\omega_1$ is \emph{internally unbounded} if
$X\cap \mathcal P_{\omega_1}(X)$ is cofinal in $\mathcal P_{\omega_1}(X)$, and is \emph{internally club} if $X\cap \mathcal P_{\omega_1}(X)$ contains a closed unbounded  subset of $\mathcal P_{\omega_1}(X)$.
It is easily seen that if $M$ is an $\omega_1$-sized elementary submodel of a  powerful model $R$, then $M$ is internally unbounded (I.U.) if and only if there is an I.U. sequence for $M$, i.e., a  $\subseteq$-increasing and $\in$-sequence $\langle M_\alpha:\alpha<\omega_1\rangle$ of countable sets  with union $M$, and that $M$ is internally club (I.C.) if and only if there is an I.C. sequence for $M$, i.e. a $\subseteq$-continuous  I.U. sequence for $M$.
	
It was shown by  Krueger \cite{Krueger2020} that if $R$ is a powerful model and $M\prec R$ with $\omega_1\subseteq M$ is an  $\omega_1$-guessing model of size $\omega_1$, then
$M$ is internally unbounded. 

Let us fix an  I.U. sequence  $\langle M_\alpha:\alpha<\omega_1\rangle$ for an I.U. $\omega_1$-guessing model $M$ of size $\omega_1$. 
For every $\alpha<\omega_1$, let 
\[
T_\alpha(M)\coloneqq \{(z,f)\in M:  f:z\cap M_\alpha\rightarrow 2\}.
\]
Let also 
\[
T(M)=\bigcup_{\alpha<\omega_1}T_\alpha(M).
\]
We order $T(M)$ by $(z,f)\leq (z',f')$ if and only if $z=z'$ and $f\subseteq f'$.

Notice that $T(M)\subseteq M$, and hence $|T(M)|\leq\omega_1$. 
Clearly $T(M)$ is a tree of height $\omega_1$. 
Assume that
$b$ is an uncountable branch through $T(M)$. Then there is $z\in M$ such that for all $(z',f')\in b$, we  have $z'=z$. 
Let 
\[
F= \bigcup \{  f :  (z,f) \in b \}.
\]
Then $F:z \cap M \mapsto 2$ is a function which is $\omega_1$-approximated in $M$. 
Since $M$ is an $\omega_1$-guessing model, there is $F^*\in M$ with $F\rest (z\cap M)=F^*\rest (z\cap M)$.
Given $F^*$, we can recover the branch $b$, so $T(M)$ has at most $\omega_1$ branches. 
By  \cref{sealing} we now have the following  \cite[Proposition 4.4]{CK2017}.

\begin{proposition}[Cox-Krueger \cite{CK2017}]\label{tree-guessing}
Suppose  $M$ is an I.U.  $\omega_1$-guessing model of cardinality $\omega_1$.
Then there is a tree $T$ of size and height $\omega_1$ with $\omega_1$ cofinal branches
such that $T$ being weakly special implies that $M$ is an indestructible $\omega_1$-guessing. 
	
\end{proposition}    
\begin{proof}[\nopunct] 
\end{proof}

\subsection*{Strong properness}

\begin{definition}[Mitchell \cite{MI2005}]\label{strong-generic} Let $\mathbb P$ be a forcing notion and $A$ a set. We say that $p\in \mathbb P$
is {\em  $(A,\mathbb P)$-strongly generic} if for every $q\leq p$ there is a condition $q\rest  A\in A$ such that any $r\in A$ with
$r\leq  q\rest  A$   is compatible with $q$. 
\end{definition}

\begin{definition}
Let $\mathbb P$ be a forcing notion, and $\mathcal S$ be a collection of sets. 
We say that $\mathbb P$ is {\em strongly proper} for $\mathcal S$ if for every $A\in \mathcal S$ and $p\in A\cap \mathbb P$,
there exists an $(A,\mathbb P)$-strongly generic condition $q\leq p$. 
\end{definition}

A forcing notion $\mathbb P$ is called strongly proper, if   for every sufficiently large  regular cardinal $\theta$ there is a club $\mathcal C$ in $\mathcal P_{\omega_1}(H_\theta)$ such that  $\mathbb P$ is 
strongly proper for  $\mathcal C$.
\begin{definition}
A forcing notion $\mathbb  P$ has the $\gamma$-approximation property if for every $V$-generic filter $ G\subseteq \mathbb P$, the pair $(V,V[G])$ has the $\gamma$-approximation property.
\end{definition}
The following are standard.
\begin{lemma}\label{SP implies approx}
Every strongly proper forcing has the $\omega_1$-approximation property.
\end{lemma}
\begin{proof}[\nopunct]
\end{proof}

\begin{lemma}
Suppose  $R$ is a powerful model, $M\prec R$ is an $\omega_1$-guessing model, and that  
$\mathbb P$ is a forcing with the $\omega_1$-approximation property.
Then $M$ remains  $\omega_1$-guessing in $V^{\mathbb P}$.  
\end{lemma}
\begin{proof}
It follows easily from \cref{transitivity of approx,guessing vs approx}.
\end{proof}

\subsection*{Magidor models}
The following definition is motivated by Magidor's reformulation of a supercompact cardinal, see \cite{MA1971}.

\begin{definition}[Magidor model, \cite{MV}]\label{Magidor models}
We say that  a model $M$  is a $\kappa$-{\em Magidor model} if, letting $\overline{M}$ be the transitive collapse of $M$ and $\pi$ the collapsing map, 
$\overline{M}=V_{\bar\gamma}$, for some $\bar\gamma <\kappa$ with ${\rm cof}(\bar\gamma) \geq \pi (\kappa)$, and $V_{\pi(\kappa)}\subseteq M$. 
If $\kappa$ is clear from the context, we omit it. 
\end{definition}

\begin{proposition}\label{Magidor-stationary} Suppose $\kappa$ is supercompact and $\mu >\kappa$ with ${\rm cof}(\mu) \geq \kappa$. 
Then the set of $\kappa$-Magidor models is stationary in $\mathcal P_{\kappa}(V_\mu)$.
\end{proposition}
\begin{proof}
See \cite[Proposition 3.21]{MV}.
\end{proof}

\subsection*{Some lemmata}

\begin{lemma}[Neeman \cite{GiltonNeeman}]\label{proper vs strong proper}
Suppose that $M$ is a $0$-guessing model which is sufficiently elementary in some transitive model $A$.
Suppose that $\mathbb P\in M$ is a forcing.
Then every  $(M,\mathbb P)$-generic condition is $(M,\mathbb P)$-strongly generic.
\end{lemma}
\begin{proof}[\nopunct]
\end{proof} 
	
We need the following 
	
\begin{lemma}\label{approx does not add guessing function}
Let $\mathbb P\in H_\theta$ be a forcing  with the $\omega_1$-approximation property. Assume that $\mathcal{P}(\mathbb P)\in H_\theta$.
Suppose  $M\prec H_\theta$ is countable and  $\mathbb P\in M$. Let $Z\in M$ and let  $f:M\cap Z\rightarrow 2$ be a function that is not guessed in  $M$.
If $p\in\mathbb P$ is $(M,\mathbb P)$-generic, then 
$$p\Vdash ``\check{f} \text{ is not guessed in } M[\dot{G}]."$$
\end{lemma}
\begin{proof}
Pick a $V$-generic filter  $G\subseteq\mathbb P$  containing $p$. Assume towards a contradiction that for some function $g:Z\rightarrow 2$ in $M[G]$, $g\rest{M\cap Z}=f$.  Suppose  $x\in [Z]^\omega\cap M$.  Then $x\subseteq M$ and hence $g\rest{x}=f\rest {x}\in V$ as   $M[G]\cap V= M$. We have that $g\rest{x} \in M$. Now by  the elementarity of $M[G]$ in $H_\theta[G]$ and the definability of $V$ in $V[G]$ (see \cite{Laver-def-V,Woodin-def-V},) we have  $g$ is countably approximated in $V$. Since $\mathbb P$ has
the $\omega_1$-approximation property, it follows that $g\in V$ and hence $g\in M$. Therefore $f$ is guessed in $M$, which is a contradiction. 
\end{proof}
	
\begin{lemma}\label{decision is guessed}
Let  $\mathbb P$ be a forcing. Assume  $\dot f:Z\rightarrow 2$ is a $\mathbb P$-name for a function and  $M$ is a model with $\dot f,Z\in M$. 
Let $p\in\mathbb P$ be an $(M,\mathbb P)$-generic condition that forces  $\dot f\rest{M\cap Z}=\check{g}$, for some function $g$.
If $g$ is guessed in $M$, then $p$ decides $\dot f$.
\end{lemma}
\begin{proof}
Suppose that $h\in M$ is a function so that $h\rest  M\cap Z=g$, we will show that $p\Vdash ``\dot f=\check h"$. Fix $q\leq p$.  We claim that there is a condition below $q$ that forces $\dot f=\check h$.
If not, the set
$$D=\{r\in\mathbb P:\exists\zeta\in Z \text{ such that } r\Vdash \check h(\zeta)\neq \dot f(\zeta)\}$$
that belongs to $M$ is pre-dense below $q$.
Since  $q$ is  $(M,\mathbb P)$-generic, 
there is $r\in D\cap M$  compatible with $q$. So there exists $\zeta\in M\cap Z$ such that $r\Vdash ``\dot f(\zeta)\neq \check h(\zeta)=\check g(\zeta)"$. This is a contradiction, since $r$ is  compatible with $q$ and  $q$ forces $\dot f(\zeta)=\check g(\zeta)$.
\end{proof}

\begin{lemma}\label{Baumgartner}
Assume  $M\prec H_\theta$ is countable and  $Z\in M$. Suppose that $z\mapsto f_z$
is a function on $[Z]^\omega$ in $M$ so that
$f_z$ is a function with $z\subseteq{\rm dom}(f_z)$. 
Let  $f: M\cap Z\rightarrow 2$ be a function  that is not guessed in $ M$.
Let $B\in M$ be  a cofinal subset of $[Z]^\omega$. Then  there is  $B^*\in M$ which is a  cofinal subset of $B$ such that for every $z\in B^*$, $f_z\nsubseteq f$.
\end{lemma}
\begin{proof}
For every $\zeta\in Z$, let
$$A_\zeta^\epsilon=\{z\in B: f_z(\zeta)=\epsilon\}, \text{ where } \epsilon=0,1.$$
Note that the sequence 
$$\langle A_\zeta^\epsilon:\zeta\in Z,\epsilon\in\{0,1\}\rangle$$
belongs to $M$.
We are done if there is some $\zeta\in Z$ such that both $A_\zeta^0$ and $A_\zeta^1$ are cofinal in $B$, as then one can find by elementarity such
$\zeta\in  M\cap Z$ and let $B^*=A_\zeta^{1-f(\zeta)}$.
Therefore, let us assume that  for every $\zeta\in  Z$, there is a unique $\epsilon\in\{0,1\}$ 
such that  $A_\zeta^\epsilon$ is cofinal in $B$.
Now define  $h$ on $ Z$ by letting $h(\zeta)$ be $\epsilon$ if and only if 
$A_\zeta^\epsilon$ is cofinal in $[Z]^\omega$. Obviously  $h$ belongs to  $M$. But 
$h\rest{M\cap Z}\neq f$,
since $f$ is not guessed in $M$. So there must exist $\zeta\in  M\cap Z$ such
that $h(\zeta)\neq f(\zeta)$, which in turn implies that $A_\zeta^{1-f(\zeta)}$ is a cofinal subset of $B$ and  belongs to $M$.
Let $B^*$ be $A_\zeta^{1-f(\zeta)}$.
\end{proof}

\begin{definition}\label{spec-poset} Let $T$ be a tree of height $\omega_1$. 
We say that $p \in \mathbb S(T)$ if 	
$p$ is a partial function,  ${\rm dom}(p)$ is a finite subset of $T$, 
$p : {\rm dom}(p)\to \omega$, and if  $x,y\in {\rm dom}(p)$ are distinct and $p(x)=p(y)$, then $x$ and $y$ are incomparable in $T$. 
The order on $\mathbb S (T)$ is the reverse inclusion. 
\end{definition} 
It is well-known that if $T$ has no uncountable branches, then $\mathbb S(T)$ has the ccc. We also need the following fact which 
was shown by Chodounsk\'y-Zapletal \cite{ChoZap}. We provide a proof for completeness. 

\begin{proposition}[Chodounsk\'y--Zapletal \cite{ChoZap}] \label{spec-poset-approx} Suppose $T$ is a tree of height $\omega_1$ without uncountable branches. 
Then $\mathbb S(T)$ has the $\omega_1$-approximation property. 
\end{proposition}

\begin{proof}
Suppose $\mu$ is a cardinal and $\dot{f}$ is an $\mathbb S(T)$-name for a function from $\mu$ to $2$ which is countably approximated in $V$. 
It suffices to show that some condition in $\mathbb S(T)$ decides $\dot{f}$.  
Fix a sufficiently large regular cardinal $\theta$ and a countable $M\prec H_\theta$ containing all the relevant objects. 
Since $\dot{f}$ is forced to be countably approximated in $V$ we can pick a condition $p \in \mathbb S(T)$ and a function $g: M \cap \mu \to 2$
such that $p$ forces that $\dot{f}\rest (M \cap \mu) = \check{g}$. Since $\mathbb S(T)$ is ccc, $p$ is $(M,\mathbb S(T))$-generic. 
Therefore, by  \cref{decision is guessed} we may assume that $g$ is not guessed in $M$. 
Let  $d= {\rm dom}(p)$, let $d_0 = d\cap M$, and $p_0 = p \rest d_0$.
Let $n=| d\setminus d_0|$. 
Let $C$ be the collection of $x\in [T \cup \mu]^\omega$ such that $d_0\subseteq x$ and $T \cap x$ is an initial segment of $T$ under $\leq_T$.
If $x\in C\cap M$, then there is a condition $p_x\leq p_0$, and a function $g_x: \mu \cap x \to 2$  such that $p_x \rest x = p_0$, $| p_x \setminus p_0 | =n$, and $p_x$ forces
that $\dot{f}\rest (x \cap \mu)$ equals $g_x$. This is witnessed by the condition $p$. By the elementarity of $M$ such a pair 
$(p_x,g_x)$ exists, for all $x\in C$. Moreover, we can pick the assignment $x\mapsto (p_x,g_x)$ in $M$. Let $d_x= {\rm dom}(p_x)$. 
By  \cref{Baumgartner} we can find $C^*\in M$, a cofinal subset of $C$ under inclusion such that $g_x\nsubseteq g$, for all $x\in C^*$. 
This implies that $p_x$ and $p$ are incompatible, for all $x\in C^*\cap M$.   
Let us fix an enumeration $d\setminus d_0=\{ t(0),\ldots, t(n-1) \}$, and, for each $x\in C^*$, an enumeration $d_x \setminus d_0 = \{ t_x(0), \ldots, t_x(n-1) \}$.

\begin{claim} Let  $B\in M$ be a cofinal subset of $C^*$, and $i,j<n$. Then there is  $B^*\in M$, a cofinal subset of $B$, 
such that for every $x\in B^* \cap M$, $t(i)$ and $t_x(j)$ are incomparable.  
\end{claim} 
\begin{proof} For $t\in T$, let ${\rm pred}_T(t)$ be the set of predecessors of $t$ in $T$, and let $h_t$ be the characteristic function of ${\rm pred}_T(t)$
as a subset of $T$. Let $h= h_{t(i)}\rest (T \cap M)$. For $x\in B$ let $h_x= h_{t_x(j)} \rest (T\cap x)$. Suppose first that $h$ is not guessed in $M$. 
Then by  \cref{Baumgartner} we can find $B^*\in M$, a cofinal subset of $B$, such that $h_x\nsubseteq h$, for all $x\in B^*$. 
This implies that if $x\in B^*\cap M$ then $t(i)$ and $t_x(j)$ do not have the same predecessors in $T\cap x$ and hence are incomparable, as desired. 
Suppose now that $h$ is guessed in $M$ and let $k\in M$ be a function such that $ k \rest (T\cap M)=h$. Since $h^{-1}(1)$ is a chain and is closed downwards in $T$, 
the same is true for $k^{-1}(1)$.  $T$ has no uncountable chains, so $k^{-1}(1)$ is countable and hence we must have $k^{-1}(1)\subseteq M$.
Moreover, $k^{-1}(1)= {\rm pred}_T(t(i))\cap M$. By the elementarity of $M$ there is $s\in M$ such that $k^{-1}(1) = {\rm pred}_T(s)$.
Now, let $B^*$ be the set of $x\in B$ such that $s\in x$. Then $B^*\in M$ is a cofinal subset of $B$. 
If $x\in B^*\cap M$, then since $t_x(j)\notin x$, it follows that  $t_x(j)\notin {\rm pred}_T(s)$. It follows that $t_x(j)$ is not below $t(i)$. 
Also, $t_x(j)$ cannot be above $t(i)$ in $T$, since then  $t_x(j)$ and all its predecessors in $T$ are in $M$ and $t(i)$ is not in  $M$.
It follows that $B^*$ is as required. 
\end{proof} 

Now, applying the above claim repeatedly for all pairs $i,j <n$ we can find a cofinal subset $B$ of $C^*$ such that, for  all $x\in B\cap M$, 
every element of $d_x\setminus d_0$ is incomparable with every element of $d\setminus d_0$.  It follows that if $x\in B\cap M$
then $p_x \cup p \in \mathbb S(T)$, which is a contradiction. 

\end{proof}

\section{The main theorems}

We  recall the guessing model principle introduced by  Viale  and Wei\ss~\cite{VW2011}.
We let	${\rm GM}(\kappa,\gamma)$ be the statement that ${\rm GM}(\kappa,\gamma, H_\theta)$ holds,
for all sufficiently large regular $\theta$, where ${\rm GM}(\kappa,\gamma, H_\theta)$ states that
\[
\mathfrak G_{\kappa,\gamma}(H_\theta) = \{ M\in {\mathcal P}_\kappa(H_\theta): M\prec H_\theta  \mbox{ and } M \mbox{ is $\gamma$-guessing} \} 
\]
is stationary in ${\mathcal P}_\kappa(H_\theta)$.

\begin{definition}[Cox--Krueger \cite{CK2017}]
$\rm SGM(\omega_2,\omega_1)$ states that for every sufficiently large   regular cardinal $\theta$ 
the following set is stationary in $\mathcal P_{\omega_2}(H_\theta)$.
\[
\mathfrak G_{\omega_2,\omega_1}(H_\theta) = \{ M\in {\mathcal P}_{\omega_2}(H_\theta): M\prec H_\theta
\mbox{ and } M \mbox{ is indestructibly $\omega_1$-guessing} \}
\]		
\end{definition}
	
Notice that it was proved in \cite{CK2017} that Suslin's Hypothesis follows from $\rm SGM(\omega_2,\omega_1)$, but not from $\rm GM(\omega_2,\omega_1)$.

Let us now recall the notion of a strongly $\gamma$-guessing model introduced and studied by the authors in \cite{MV}.

\begin{definition}\label{Strong guessing}
Let   $\gamma \leq \kappa$  be regular uncountable cardinals. A model  $M$ of cardinality $\kappa^+$
is called {\em strongly $\gamma$-guessing} if it is the union of an $\in$-increasing chain $\langle M_\xi : \xi <\kappa^+\rangle$
of $\gamma$-guessing models of cardinality $\kappa$ such that $M_\xi= \bigcup \{ M_\eta : \eta < \xi\}$, for every $\xi$ of cofinality $\kappa$. 
\end{definition}
It was proved in Remark 2.14 of \cite{MV} that every strongly $\gamma$-guessing model is $\gamma$-guessing.
	
\begin{definition}\label{GM+}
${\rm GM}^+(\kappa^{++},\gamma)$ is the statement that ${\rm GM}^+(\kappa^{++},\gamma, H_\theta)$ holds,
for all sufficiently large regular $\theta$, where ${\rm GM}^+(\kappa^{++},\gamma, H_\theta)$ states that
\[
\mathfrak G^+_{\kappa^{++},\gamma}(H_\theta) = \{ M\in {\mathcal P}_{\kappa^{++}}(H_\theta): M\prec H_\theta  \mbox{ and } M \mbox{ is strongly $\gamma$-guessing} \}. 
\] 
is stationary in ${\mathcal P}_{\kappa^{++}}(H_\theta)$.
\end{definition}

\begin{definition}
A model  $M$ of cardinality $\omega_2$
is {\em indestructibly strongly $\omega_1$-guessing} if it is the union of an increasing chain $\langle M_\xi : \xi <\omega_2\rangle$
of indestructibly $\omega_1$-guessing models of cardinality $\omega_1$ such that 
$M_\xi= \bigcup \{ M_\eta : \eta < \xi\}$, for every $\xi$ of cofinality $\omega_1$. 
\end{definition}

\begin{definition}
The principle ${\rm SGM}^+(\omega_3,\omega_1)$ states that $\mathfrak S^+_{\omega_3,\omega_1}(H_\theta)$ is stationary, 
for all large enough $\theta$, where
$$\mathfrak G^+_{\omega_3,\omega_1}(H_\theta)\coloneqq 
\{ M\in \mathcal P_{\omega_3}(H_\theta) : M\prec H_\theta \text{ is   indestructibly  and strongly } \omega_1\text{-guessing}\}.$$
\end{definition}

\begin{definition}
For a regular cardinal $\kappa$, we let  $\rm{ MP(\kappa^+)}$ denote the statement that		
every forcing that adds a new subset of $\kappa^+$ whose initial segments are in the 
ground model  collapses some cardinal $\leq 2^{\kappa}$.
\end{definition} 
Todor\v{c}evi\'{c} \cite{Todorcevic82} showed that if every tree of size and height $\omega_1$ with at most $\omega_1$ cofinal branches  is weakly special and $2^{\aleph_0} < \aleph_{\omega_1}$ 
then $\rm MP(\omega_1)$ holds.  The principle was also studied  by Golshani and Shelah in \cite{GolShe}, where they showed that ${\rm MP}(\kappa^+)$ is consistent, for every regular cardinal $\kappa$.

\begin{proposition}[Cox--Krueger \cite{CK2017}]
If  $\rm SGM(\omega_2,\omega_1)$ and $2^{\aleph_0}< \aleph_{\omega_1}$ then  ${\rm MP}(\omega_1)$ holds.
\end{proposition}
\begin{proof}[\nopunct]
\end{proof} 
	
We shall prove that $\rm SGM^+(\omega_3,\omega_1)$ implies ${\rm MP}(\omega_2)$, and since ${\rm SGM}(\omega_2,\omega_1)$ follows from
${\rm SGM}^+(\omega_3,\omega_1)$, we get the simultaneous consistency of ${\rm MP}(\omega_1)$ and ${\rm MP}(\omega_2)$.

\begin{theorem}\label{nice application}
Suppose that $V\subseteq W$ are transitive models of $\rm ZFC$. 
Assume  that $\rm SGM^+(\omega_3,\omega_1)$ 
and $2^{\omega_1}<\aleph_{\omega_{2}}$ hold in $V$. 
Suppose that $W$ has a  subset of $\omega^V_2$ that does not belong to $V$.
Then either $\mathcal P^V(\omega_1)\neq \mathcal P^W(\omega_1)$ or  some $V$-cardinal  $\leq 2^{\omega_1}$ is no longer a cardinal in $W$.
\end{theorem}
\begin{proof}
Let $x\in W\setminus V$ be a  subset of $\omega^V_2$, and suppose that $\mathcal P^V(\omega_1)=\mathcal P^W(\omega_1)$.
We will show that  some cardinal  $\leq 2^{\omega_1}$ is no longer a cardinal in $W$.
By $\mathcal P^V(\omega_1)= \mathcal P^W(\omega_1)$, 
every initial segment of $x$ belongs to $V$. Let  \[
\mathfrak X=\{x\cap\gamma:\gamma<\omega_2\}.
\]
Note that  $\mathfrak X$ is  bounded in $V$.
Assume towards a contradiction that every $V$-cardinal  $\leq 2^{\omega_1}$ remains a  cardinal in $W$.
Working in $W$,  let $\mu\geq\omega_2$ be the least cardinal so that there is a set $M$ in $V$ of cardinality $\mu$ such that $M\cap \mathfrak X$ is of size $\omega_2$. Thus $\mu\leq 2^{\omega_1}$.
\begin{claim}
     $\mu=\omega_2$. 
\end{claim}
\begin{proof}
Assume towards a contradiction that $\mu>\omega_2$ and $M$ is a witness to it,  then one can work in $V$ and write  $M$ as the union of an increasing sequence $\langle M_\xi:\xi<{\rm cof^{V}}(\mu)\rangle$ of subsets of $M$ in $V$ whose size are less than $\mu$.
Since $\mu\leq 2^{\omega_1}<\aleph_{\omega_2}$
and every cardinal $\leq 2^{\omega_1}$ is a cardinal in $W$, we have ${\rm cof}^W(\mu)={\rm cof}^V(\mu)\neq\omega_2$. 
Thus either $\mu$ is of cofinality at most $\omega_1$,
and then by the pigeonhole principle, there is $\xi<{\rm cof}(\mu)$ such that $|M_\xi\cap\mathfrak X|=\omega_2$, or $\mu$
is regular, and then there is some $\xi<\mu={\rm cof}(\mu)$ such that  $M\cap \mathfrak X\subseteq M_\xi$. In either case, we
get a  contradiction since $|M_\xi|<\mu$.
\end{proof}
By the above claim, $\mu=\omega_2$. Let $M$ be a witness for $\mu=\omega_2$, and let $\mathfrak X'=M\cap\mathfrak X$. Notice that $V\models |M|=\omega_2$. Since $M$ is in $V$ and  $V$ satisfies  $\rm SGM^{+}(\omega_3,\omega_1)$, one can
cover $M$ by an indestructibly strongly $\omega_1$-guessing model $N$ of size $\omega_2$.
Working in $W$,  $x$ is countably approximated in $N$, since if $\gamma\in N\cap \omega_2$, then there is
$\gamma'>\gamma$ in $N$ such that $x\cap\gamma'\in\mathfrak X'\subseteq N$, and hence $x\cap\gamma\in N$. On the other hand, $N$ is a guessing model in $W$ by $\rm SGM^+(\omega_3,\omega_1)$ in $V$ and both $\omega_1$ and $\omega_2$ are cardinals in $W$.
Thus $x$ is guessed in $N$, but then $x$ must be in $N$ since $\omega_2 \subseteq N$. 
Therefore, $x$ is in $V$, which is a contradiction!
\end{proof}

The following corollaries are immediate.
\begin{corollary}
Suppose that $V\subseteq W$ are transitive models of $\rm ZFC$. Assume that $\rm SGM^+(\omega_3,\omega_1)$, $2^{\omega_0}<\aleph_{\omega_{1}}$ and $2^{\omega_1}<\aleph_{\omega_{2}}$ hold in $V$.
Suppose that $W$ has a new subset of $\omega^V_2$.
Then either $W$ contains a real which is not in $V$ or some cardinal $\leq 2^{\omega_1}$ in $V$ is collapsed  in $W$. 
\end{corollary}
\begin{proof}[\nopunct]
\end{proof}
\begin{corollary}
Assume   $2^{\aleph_0}<\aleph_{\omega_1}$ and $2^{\aleph_1}<\aleph_{\omega_2}$ hold. 
Then $\rm SGM^+(\omega_3,\omega_1)$ implies both ${\rm MP}(\omega_1)$ and ${\rm MP}(\omega_2)$.
\end{corollary}
\begin{proof}[\nopunct]
\end{proof}
Our main theorem reads as follows.
\begin{theorem}[Main Theorem]\label{main-theorem}
Assume there are two supercompact cardinals. There is a generic extension $V[G]$ of $V$ in which  
${\rm SGM}^+(\omega_3,\omega_1)$  and $2^{\aleph_0}=2^{\aleph_1}=2^{\aleph_2}=\aleph_3$ hold.
\end{theorem}
\begin{proof}[\nopunct]
\end{proof}
\section{An overview of virtual models}
	
We review the theory of virtual models from the authors' paper \cite{MV}.

\subsection*{The general setting}

We consider the language $\mathcal L$ obtained by adding a single constant symbol $c$ and a predicate $U$ to the standard language of set theory. 
Let us say that an $\mathcal L$-structure $\mathcal A=(A,\in, \kappa, U)$ is \emph{suitable}
if $A$ is a transitive set  and $\mathcal A\models``\ZFC"$ in the expanded language, 
where $\kappa=c^{\mathcal A}$ is an uncountable regular in $\mathcal A$. 
For a suitable structure $\mathcal A$ and an ordinal  $\alpha\in A\setminus\kappa$,
 let $\mathcal A_\alpha=(A\cap V_\alpha,\in,\kappa, U\cap V_\alpha)$. Finally, we let
\[
E_{\mathcal A}=\{ \alpha \in {\rm ORD}^A : \mathcal A_\alpha \prec \mathcal A \}. 
\]
Note that $E_A$ is a closed subset of $\ORD^{\mathcal A}$. It is not definable in $\mathcal A$, but $E_{\mathcal A}\cap \alpha$
is uniformly definable in $\mathcal A$ with parameter $\alpha$, for each
$\alpha \in E_{\mathcal A}$. 
For $\alpha \in E_{\mathcal A}$,  we let ${\rm next}_{\mathcal A}(\alpha)$ be the least ordinal in $E_{\mathcal A}$ above $\alpha$,
if such an ordinal exists. Otherwise, we leave ${\rm next}_{\mathcal A}(\alpha)$ undefined.
We shall often abuse  notation and refer to the structure $\mathcal A$ by $A$.
	
\begin{lemma}\label{alphasharpinE}  
Suppose $M$ is an elementary submodel of  a suitable structure $A$,  and that $\alpha \in E_A$. 
If $(M\cap \ORD^A)\setminus\alpha\neq\varnothing$, then ${\rm min}(M\cap\ORD^A\setminus\alpha)\in E_A$.
\end{lemma}
	
\begin{proof}
See  \cite[Lemma 1.1]{Velisemiproper} or  \cite[Lemma 3.1]{MV}.
\end{proof}
	
If $M$ is a submodel of a suitable
structure $A$ and $X$ is a subset of $A$, let the operation ${\rm Hull}$  be defined by
\[
{\rm Hull}(M,X) =\{ f(\bar{x}) : f\in M, \bar{x}\in X^{<\omega}, f \mbox{ is a function, and } \bar{x}\in \dom (f)\}.
\]
	
For $\alpha \in E_A$, we let ${\mathscr A}_\alpha$ denote the set of all
transitive $\mathcal L$-structures $\mathcal H\in A$ which are elementary extensions of $\mathcal A_\alpha$ and have the same cardinality as $A_\alpha$.
Note that if $\mathcal H \in \mathscr A_\alpha$ and $\alpha \in \mathcal H$, then $E_{\mathcal H}\cap \alpha=E_A\cap \alpha$.
For $\mathcal H\in \mathscr A_\alpha$, we will refer to $A_\alpha$ as the {\em standard part} of $\mathcal H$. Note that if $\mathcal H$ has nonstandard elements, then $\alpha \in E_{\mathcal H}$.

\begin{lemma}\label{hull-elementary} 
Suppose $M$ is an elementary submodel of  a suitable structure $A$,  and that $X\subseteq A$.
Let $\delta= \sup (M\cap \ORD^A)$, and suppose that $X\cap A_\delta$ is nonempty.
Then ${\rm Hull}(M,X)$ is the least elementary submodel of $A$ containing $M$ and $X\cap A_\delta$ as subsets.
\end{lemma}
\begin{proof}
See   \cite[Lemma 1.2]{Velisemiproper} or \cite[Lemma 3.3]{MV}.
\end{proof}
\subsection*{Virtual models in \texorpdfstring{$V_\lambda$}{}}
Assume that $\kappa<\lambda$ are supercompact and  inaccessible cardinals, respectively. Consider the $\mathcal L$-structure \[
V_\lambda\coloneqq(V_\lambda,\in,\kappa, U).
\]
Let $E=E_{V_\lambda}$ and ${\rm next}(\alpha)={\rm next}_{V_\lambda}(\alpha)$.
	
\begin{definition}
Suppose $\alpha \in E$. We let $\mathscr{V}_\alpha$ denote the collection of all substructures $M$ of $V_\lambda$ of
size less than $\kappa$ so that
if we let $A={\rm Hull}(M,V_\alpha)$, then $A\in \mathscr A_\alpha$ and $M$ is an elementary 
submodel of $ A$. The members of $\mathscr{V}_\alpha$ are called $\alpha$-{\em models}. 
\end{definition}
	
We write $\mathscr{V}_{< \alpha}$ for $\bigcup \{\mathscr{V}_\gamma : \gamma \in E \cap \alpha\}$ and  $\mathscr V$ for $\mathscr{V}_{< \lambda}$.
The collections $\mathscr{V}_{\leq \alpha}$ and $\mathscr{V}_{\geq \alpha}$ are defined in the obvious way. 
For $M\in \mathscr V$,  $\eta(M)$ denotes the unique ordinal $\alpha$ such that $M\in \mathscr V_\alpha$.
Note that if $M \in \mathscr V_\alpha$, then $\sup(M \cap \ORD)\geq \alpha$.
In general, $M$ is not elementary in $V_\lambda$, in fact, this only happens if $M \subseteq V_\alpha$.
In this case, we will say that $M$ is a {\em standard $ \alpha$-model}.
We refer to the members of $\mathscr V$ as
{\em virtual models}. We also refer to members of $\mathscr V^A$, for some
suitable structure $A$ with $A\subseteq V_\lambda$, as {\em general virtual models}.

Suppose $M,N\in \mathscr V$ and $\alpha \in E$.
We say that an isomorphism $\sigma:M\rightarrow N$ is an $\alpha$-{\em isomorphism}
if it has an extension to an isomorphism $\bar{\sigma}:{\rm Hull}(M,V_\alpha)\rightarrow {\rm Hull}(N,V_\alpha)$.
We say that $M$ and $N$ are $\alpha$-{\em isomorphic} and write $M\cong_{\alpha}N$ if
there is an $\alpha$-isomorphism between them. Note that   $\sigma$ and $\bar\sigma$, if they exist,  are unique.

It is clear that $\cong_{\alpha}$ is an equivalence relation, for every $\alpha\in E$. If $\alpha <\beta$ are in $E$, then
for every $\beta$-model $M$ there is a canonical representative of the $\cong_\alpha$-equivalence
class of $M$ which is an $\alpha$-model.

\begin{definition}\label{projection-models-def} 
Suppose $\alpha$ and $\beta$ are members of  $E$
and $M$ is a $\beta$-model. Let $\overline{{\rm Hull}(M,V_\alpha)}$ be the transitive collapse of
${\rm Hull}(M,V_\alpha)$, and let $\pi$ be the collapsing map. We define  $M\rest \alpha$ to be $\pi[M]$,
i.e., the image of $M$ under the collapsing map of ${\rm Hull}(M,V_\alpha)$.
\end{definition}
	
Note  that if $A\in \mathscr A_\alpha$ then $\mathscr V_\alpha^A \subseteq \mathscr V_\alpha$. 
Therefore, if $A,B\in \mathscr A_\alpha$, $M \in \mathscr V^{A}$, and $N \in \mathscr V^B$, we can 
still write $M \cong_\alpha N$ if $M \rest \alpha = N \rest \alpha$. It is straightforward to check that if 
$\alpha \leq \beta$ are in $E$, and $M\in \mathscr V$. Then
$(M   \rest   \beta)  \rest   \alpha = M  \rest   \alpha$.
For each $\alpha\in E$, the virtual version of the membership relation is defined as follows.

Suppose $M,N\in \mathscr V$  and $\alpha \in E$. 
We write $M\in_\alpha N$ if there is $M'\in N$ with $M' \in \mathscr V^N$ such that $M'\cong_\alpha M$. 
If this is the case, we say that $M$ is $\alpha$-in $N$.

\begin{lemma}\label{projection-membership} Let $\alpha\leq \beta$ be in $ E$.
Suppose $M,N \in \mathscr V_{\geq \beta}$ and $M\in_\beta N$. 
Then $M  \rest   \alpha \in_\alpha N  \rest   \alpha$.
\end{lemma}
\begin{proof}
See  \cite[Proposition 3.15]{MV}.
\end{proof}

\begin{definition}
For $\alpha \in E$, we let $\mathscr C_\alpha$ denote the collection of countable models in  $\mathscr V_\alpha$, and   we  let $\mathscr U_{\alpha}$ be the collection of all $M\in \mathscr V_{\alpha}$ that are $\kappa$-Magidor models. 
\end{definition}
The collections $\mathscr C_{<\alpha}$, $\mathscr C_{\leq\alpha}$,  $\mathscr C_{\geq \alpha}$, $\mathscr U_{<\alpha}$, $\mathscr U_{\leq\alpha}$, and $\mathscr U_{\geq \alpha}$ are defined similarly. We write $\mathscr C$ for $\mathscr C_{<\lambda}$, and $\mathscr U$ for $\mathscr U_{< \lambda}$,
$\mathscr C_{\rm st}$ for the collection of standard models in $\mathscr C$, and $\mathscr U_{\rm st}$ for the standard models in $\mathscr U$. Note that  both classes $\mathscr C$ and $\mathscr U$  are closed under projections.

\begin{lemma}\label{cstationaryinV}
$\mathscr C_{\rm st}$ contains a club in ${\mathcal P}_{\omega_1}(V_\lambda)$, and  $\mathscr U_{\rm st}$ is stationary in $\mathcal P_{\kappa}(V_\lambda)$.
\end{lemma}
\begin{proof}
See \cite[Propositions 3.19 and  3.24]{MV}.
\end{proof}

\begin{definition}\label{def-active}
Let $M\in \mathscr V$. We say that $M$ is {\em active} at $\alpha \in E$ if 
$\eta (M)\geq \alpha$ and ${\rm Hull} (M,V_{\kappa_M})\cap E \cap \alpha$ is unbounded in $E\cap \alpha$,
where $\kappa_M = \sup (M\cap \kappa)$. We say that $M$ is {\em strongly active} at $\alpha$ if $\eta(M)\geq \alpha$
and $M\cap E\cap \alpha$ is unbounded in $E\cap \alpha$. 
\end{definition}
	
One can easily see that a Magidor model is active at $\alpha$ if and only if it is strongly active at $\alpha$.

\begin{notation}\label{a(M)} 
For a model $M\in \mathscr V$, let $a(M)= \{ \alpha \in E : M \mbox{ is active at } \alpha \}$ and $\alpha (M)=\max(a(M))$.
\end{notation}
Note that $a(M)$ is a closed subset of $E$ of size at most $|{\rm Hull}(M,V_{\kappa_M})|$.
We now  recall the definition of a {\em meet} from \cite{MV} which is the virtual version of intersection  defined only for  two models of different types. 
Suppose $N\in \mathscr U$ and $M\in \mathscr C$. Let $\overline{N}$ be the transitive collapse of $N$,
and let $\pi$ be the collapsing map. Note that if $\overline{N}\in M$, then $\overline{N}\cap M$ is a countable elementary submodel 
of $\overline{N}$. Then $\overline{N}\cap M \in \overline{N}$ since $\overline{N}$ is closed under countable sequence.
Note that $\pi^{-1}(\overline{N}\cap M)= \pi^{-1}[\overline{N}\cap M]$, and this model is elementary in $N$.

\begin{definition}\label{meet-def} Suppose $N\in \mathscr U$ and $M \in \mathscr C$. Let $\alpha = \max (a(N)\cap a(M))$. 
We will define $N \land M$ if $N\in_\alpha M$. Let $\overline{N}$ be the transitive collapse of $N$, and let $\pi$
be the collapsing map. 
Set
\[ 
\eta = \sup (\sup (\pi^{-1}[\overline{N}\cap M]\cap {\rm ORD})\cap E\cap (\alpha +1)).
\]
We define the meet of $N$ and $M$ to be $N\land M=\pi^{-1}[\overline{N}\cap M]\restriction\eta$.

\end{definition}

\begin{remark}\label{remark on meet properties}
It was proved in \cite{MV} that $N\land M\in \mathscr C_\eta$ is an $\eta$-model, and, furthermore, if $N\land M$ is active at $\gamma$, then $(N\land M)\rest  \gamma=N\rest  \gamma\land M\rest  \gamma$.
More intersection-like properties of $\land$ are found  on \cite[Pages 14--16]{MV}.
\end{remark}

\begin{notation}
Let $\alpha \in E$ and let  $\mathcal M$ be a set of virtual models.
We let \[
\mathcal M\rest\alpha=\{M\rest\alpha:M\in\mathcal M\} \mbox{ and } 
\mathcal M^\alpha=\{M\rest\alpha:~M\in \mathcal M \text{ is active at } \alpha\}.
\]
\end{notation}

\begin{definition}
Let $\alpha \in E$ and let $\mathcal M$ be a subset of $\mathscr U \cup \mathscr C$.
We say $\mathcal M$ is an $\alpha$-{\em chain} if for all distinct  $M,N\in\mathcal M$, either $M\in_\alpha N$ or $N\in_\alpha M$,
or there is a $P\in\mathcal M$ such that either $M\in_\alpha P\in_\alpha N$  or $N\in_\alpha P\in_\alpha  M$.
\end{definition}
	
\begin{lemma}\label{finite-chain}
Suppose $\alpha \in E$ and $\mathcal M$ is a finite subset of  $\mathscr U \cup \mathscr C$.
Then $\mathcal M$ is an $\alpha$-chain if and only if there is an enumeration 
$\{ M_i : i <n \}$ of $\mathcal M$ such that $M_0\in_\alpha \cdots \in_\alpha M_{n-1}$. 
\end{lemma}
\begin{proof}
See    \cite[Proposition 3.41]{MV}.
\end{proof}

\section{Virtual models as side conditions}
	
We recall the definition of our pure side conditions forcing with decorations introduced in \cite{MV}. We give its basic properties and refer the reader to the above paper for the proofs. Recall that $\kappa<\lambda$ are supercomapct and inaccessible, respectively. 

\begin{definition}[pure side conditions]\label{MF}
Suppose $\alpha\in E$. We say that $p=\M_p$ belongs to $\mathbb M^\kappa_\alpha$ if:
\begin{enumerate}
\item $\M_p\subseteq\mathscr C_{\leq \alpha}\cup \mathscr U_{\leq \alpha}$  is finite and closed under meets,  and
\item $\M_p^\delta$ is a $\delta$-chain, for all $\delta \in  E\cap (\alpha +1)$.
\end{enumerate}
We let $\M_q\leq \M_p$ if for all $M\in \M_p$ there is $N\in \M_q$ such that  $N \rest{\eta(M)}=M$.
Finally, let $\mathbb M^{\kappa}_\lambda=\bigcup \{ \mathbb M^{\kappa}_\alpha: \alpha \in E\}$ with the same order.
\end{definition}

Suppose $\M_p\in \mathbb M^\kappa_\lambda$. 
Let 
\[
\mathcal L(\M_p)= \{ M\rest \alpha: M\in \M_p \text{ and } \alpha \in a(M)\}.
\]
We say that $M\in \mathcal L(\M_p)$ is $\M_p$-{\em free}
if  every $N\in \M_p$ with $M\in_{\eta(M)}N$  is strongly active at $\eta(M)$. 
Let $\mathcal F(\M_p)$ denote the set of all $M\in \mathcal L(\M_p)$ that are $\M_p$-free. 
Note that a model $M\in \mathcal L(\M_p)$ that is not  $\M_p$-free is not $\M_q$-free, for any  $\mathcal M_q\leq \mathcal M_p$.
	
\begin{definition}[side conditions with decorations]\label{PF} Suppose $\alpha\in E\cup \{ \lambda\}$. We say that a pair $p=(\M_p,d_p)$ belongs to $\mathbb{P}^\kappa_\alpha$ if 
$\M_p\in \mathbb M^\kappa_\alpha$, $d_p$ is a finite partial function from $\mathcal F(\M_p)$ to ${\mathcal P}_{\omega}(V_\kappa)$ such that
\medskip
\begin{itemize}
\item[$(\divideontimes)$]  \hspace{5mm} if $M \in \dom (d_p)$, $N\in \M_p$, and $M\in_{\eta(M)} N$, then $d_p(M)\in N$. 
\end{itemize}
\medskip
We say that $q\leq p$ if $\M_q \leq \M_p$ in $\mathbb M^\kappa_\alpha$, and, for every $M\in \dom(d_p)$, there exists some $\gamma \in E\cap (\eta(M)+1)$
with $M\rest \gamma \in \dom(d_q)$ such that $d_p(M)\subseteq d_q(M\rest \gamma)$. 
\end{definition}
	
We call $d_p$ the decoration of $p$.
The order on $\mathbb{P}^\kappa_\lambda$ is transitive. We will say that $q$ is {\em stronger} than $p$
if $q$ forces that $p$ belongs to the generic filter, in other words, any $r\leq q$ is compatible with $p$.
We write $p\sim q$ if each of $p$ and $q$ is stronger than the other. We identify equivalent conditions, often without saying so.
Our forcing does not have the greatest lower bound property, but if $p$ and $q$ have a greatest lower bound, we will denote it by $p\land q$. To be precise, we should refer to $p\land q$ as the $\sim$-equivalence class of a greatest lower bound, but we ignore this point as it should not cause any confusion.
Note that if $p\in \mathbb{P}^\kappa_\alpha$ and $M\in \M_p$ is a $\delta$-model that is not active at $\delta$, we may replace $M$ with $M\rest {\alpha(M)}$ and get an equivalent condition. Thus, if $\alpha \in E$ and $\cof(\alpha)\geq \kappa$, then $\mathbb{P}^\kappa_\alpha$ is forcing equivalent to
$\bigcup \{ \mathbb{P}^\kappa_\gamma : \gamma \in E\cap \alpha\}$. Suppose $\alpha, \beta \in E$ and $\alpha \leq \beta$. For every $p\in\mathbb{P}^{\kappa}_\beta$, we let $\M_{p\rest\alpha} =\M_p\rest \alpha$ and $d_{p\rest \alpha}= d_p\rest {\mathcal F(\M_p\rest \alpha)}$. It is easy to see that $p\rest \alpha= (\M_{p\rest \alpha},d_{p \rest \alpha})\in \mathbb{P}^{\kappa}_\alpha$.
The following is straightforward.
	
\begin{lemma}\label{CS}
Suppose that $\alpha,\beta \in E$ with $\alpha \leq \beta$. Let $p\in\mathbb{P}^{\kappa}_\beta$ 
and let $q\in\mathbb{P}^{\kappa}_\alpha$ be such that $q \leq p\rest \alpha$.
Then there exists $r\in\mathbb{P}^{\kappa}_\beta$ such that $r\leq p,q$.
\end{lemma}

\begin{proof} We let $\M_r = \M_p \cup \M_q$. Note that $\M_r$ is closed under meets. 
We define $d_r$  by letting  $d_r(M)=d_q(M)$ if $M\in \dom(d_q)$, and $d_r(M)= d_p(M)$ if $M\in \dom(d_p)$ with $\eta(M) > \alpha$. 
It is straightforward that $r$ is as required. 
\end{proof}
	
\begin{remark} The condition $r$ from the previous lemma is the greatest lower bound of $p$ and $q$, so we will let $r \coloneqq p \land q$. 
\end{remark}
	
\begin{definition}\label{M on top 1}
Assume $p\in \mathbb{P}^{\kappa}_\lambda$, and that $M\in \mathscr C \cup \mathscr U$ be such that $p\in M$.
Let $\M_{p^M}$ be the closure of $\M_p\cup \{ M\}$ under meets, and let $d_{p^M}$ be defined on
$${\rm dom}(d_{p^M})=\{N\rest\delta: N\in {\rm dom}(d_p) \text{ and } \delta={\rm sup}(M\cap\eta(N))\},$$
by letting $d_{p^M}(N\rest\delta)=d_p(N)$.
		
\end{definition}

\begin{lemma}\label{topM}
$p^M$ is the weakest condition extending $ p$ with $M\in \M_{p^M}$.
\end{lemma}
\begin{proof}
\cite[Lemma 4.11]{MV}.
\end{proof}
	
\begin{notation}
For virtual models $N,M$, we set $\alpha (N,M)= \max (a(N)\cap a(M))$. 
\end{notation}
We are now about to give the restriction of a condition to a given model. We start with Magidor models.
	
\begin{definition}\label{restriction-magidor} Let $p\in \mathbb{P}^\kappa_\lambda$. Assume that $M\in \mathcal L(\M_p)$ is a Magidor model. 
For $N\in \M_p$, we let $N\rest M = N \rest {\alpha(N,M)}$ if $\kappa_N < \kappa_M$, 
otherwise $N\rest M$ is undefined. 
Let  $p\rest M$ be defined by
\[
\M_{p \rest M} = \{ N \rest M : N \in \M_p\} \mbox{ and } d_{p \rest  M}= d_p \rest{(\dom(d_p)\cap M)}.
\]
\end{definition}
	
\begin{proposition}\label{rest-magidor-prop} Suppose $p\in \mathbb{P}^\kappa_\lambda$ and $M\in \mathcal L(\M_p)$ is a Magidor model. 
Then $p\rest  M \in \mathbb{P}^\kappa_\lambda \cap M$ and $p\leq p\rest  M$.  Moreover, if $q\in M\cap\mathbb{P}^\kappa_\lambda$ extends $p\rest  M$.
Then $q$ is compatible with $p$ and the meet $p\land q$  exists.
\end{proposition}
\begin{proof}
See  \cite[Lemmata 4.13 and 4.14]{MV}.
\end{proof}
As a corollary, we have that  the forcing $\mathbb{P}^\kappa_\lambda$ is strongly proper for $\mathscr U$.
We now define an analogue of $p\rest  M$ for  countable models $M\in \mathcal L(\M_p)$.
Now suppose $p\in \mathbb{P}^\kappa_\alpha$ and $M\in \M_p$ is either a standard countable  model or a countable model with $\alpha\in M$.  
Suppose $N\in \M_p$ and  $N\in_\gamma M$, where  $\gamma=\max(a(M)\cap a(N))$. If $\gamma\in M$, then $N\rest  \gamma\in M$, but it may be that $\gamma \notin M$, but then, if we let $\gamma^*= \min (M\cap \lambda\setminus \gamma)$, we have
that $\gamma^*\in E\cap M$ by \cref{alphasharpinE}.
Thus there is a $\gamma^*$-model $N^*\in M$ 
which is $\gamma$-isomorphic to $N\rest  \gamma$.  Moreover, such $N^*$ is unique\footnote{ See the paragraph above Definition 4.21 of \cite{MV}}.

\begin{definition}\label{restriction-countable} 
Let $p\in \mathbb{P}^\kappa_\alpha$. Assume that  $M\in \mathcal L(\M_p)$ is  a countable model that is either standard or contains $\alpha$. 
Suppose that $N\in \M_p$, and let $\gamma = \alpha (M,N)$. If $N\in_\gamma M$, we let $\gamma^*= \min (M\cap \lambda\setminus \gamma)$.
We define $N\rest  M$ to be the unique $\gamma^*$-model  $N^*\in M$ such that $N^* \cong_\gamma N$.  
Otherwise, $N\rest  M$ is undefined. Let
\[
\M_ {p\rest  M} = \{ N\rest  M : N\in \M_p\} \text{ and}
\]
\[
\dom(d_{p \rest  M})= \{ N\rest  M: N\in \dom(d_p) \text{ and } N\in_{\eta (N)}M \}.
\]
If $N\in \dom(d_p)$ and $N\in_{\eta(N)} M$, we then let $d_{p\rest  M}(N\rest M)= d_p(N)$.
Let also \[
p\rest M \coloneqq (\M_{p\rest  M},d_{p \rest  M}).\] 
\end{definition}
	
\begin{remark} Suppose $N\in \dom (d_p)$ and let $\eta= \eta(N)$. If $N\in_\eta M$ then $M$ is strongly active at $\eta$ since $N$ is $\M_p$-free. 
If $\eta \in M$ then we put $N$ in $\dom(d_{p\rest M})$ and keep the same decoration at $N$. 
If $\eta  \notin M$ we lift $N$ to the least level $\eta^*$ of $M$ above $\eta$, we put the resulting model $N^*$ in $\dom(d_{p\rest  M})$
and copy the decoration of $N$ to $N^*$.  If $P\in \M_p$ is such that $P\rest \eta =N$ then $(P\rest M)\rest {\eta^*} = N^*$.
Moreover, we can recover $N$ from $N^*$  as $N^*\rest {\sup (\eta^*\cap M)}$. Thus, the function $d_{p\rest  M}$ is well-defined. 
Note also that $p\rest M \in M$. 
\end{remark}
	
\begin{proposition}\label{SP}
Let $p\in \mathbb{P}^\kappa_\alpha$ and let $M\in \mathscr C$ such that $\alpha \in M$ and $M\rest \alpha \in \mathcal M_p$. 
Then $p\rest  M\in \mathbb{P}^\kappa_{\alpha}$, and for any $q\in \mathbb{P}^\kappa_{\alpha}\cap M$ with $q\leq p\rest M$, $p$ and $q$ are compatible, and  
the meet $p\land q$ exists. 
\end{proposition}
\begin{proof}
See  \cite[Lemma 4.26]{MV}.
\end{proof}
	
Note that if $p,q\in\mathbb{P}^\kappa_\lambda$ are such that $q\leq p$ and $M\in \M_p$ then  $q\rest M \leq p\rest M$. It follows from \cref{topM,SP} that $\mathbb{P}^\kappa_\lambda$ is strongly proper for $\mathscr C_{\rm st}$, and that  $\mathbb{P}^\kappa_\alpha$ is strongly proper for the collection of all models $M\in\mathscr C$ with $\alpha\in M$.

For a filter $F$  in $\mathbb{P}^\kappa_\lambda$, we set  $\mathcal M_{F}\coloneqq\bigcup\{\mathcal M_p:~p\in F\}$, and for a $V$-generic filter $G$ in $\mathbb{P}^\kappa_\lambda$, we let $G_\alpha=G\cap\mathbb{P}^\kappa_\alpha$, for all $\alpha\in E$.  The following is easy. 
	
\begin{lemma}\label{chainfilter}
For every $\delta \in E$ with ${\rm cof}(\delta)<\kappa$, $\mathcal M_G^\delta$ is a $\delta$-chain.
\end{lemma}
\begin{proof}[\nopunct]
\end{proof}
	
\begin{proposition}\label{continuity-Magidor} 
Let $\delta\in E$  with ${\rm cof}(\delta)<\kappa$.  Suppose $M\in \M_G^\delta$ is a Magidor model that is not the least model in $\M_G^\delta$. Then  
$$M\cap V_\delta = \bigcup \{ Q\cap V_\delta : Q\in_\delta M \mbox{ and } Q\in \M_G^\delta\}.$$ 
\end{proposition}
\begin{proof}
See  \cite[Proposition 4.32]{MV}.
\end{proof}

\begin{proposition}\label{collapse with pure side condition}
The forcing
$\mathbb{P}^{\kappa}_\lambda$  collapses all uncountable cardinals 
below $\kappa$  to $\omega_1$ and those between $\kappa$ and $\lambda$ to $\kappa$. 
\end{proposition}
\begin{proof}
See \cite[Theorems 4.33 and 4.34]{MV}.
\end{proof}
	
Notice  $\mathbb{P}^{\kappa}_\lambda$ is $\lambda$-c.c., see  \cref{lambdacc-iteration}.

\begin{definition}\label{def-clubs} 
Suppose that $G\subseteq \mathbb{P}^\kappa_\lambda$ is $V$-generic, and that  $\alpha \in E$ is of cofinality less than $\kappa$. 
We let $C_\alpha(G)= \{ \kappa_M : M\in \M_G^\alpha\}$. 
\end{definition}

\begin{proposition}\label{addingclub}
Let $G$ be a $V$-generic filter over $\mathbb{P}^\kappa_\lambda$. Then $C_\alpha(G)$ is a club in $\kappa$, for all $\alpha \in E$
of cofinality ${<}\kappa$.  Moreover, if $\alpha < \beta$ then $C_\beta(G)\setminus C_\alpha(G)$ is bounded in $\kappa$. 
\end{proposition}
\begin{proof}
See   \cite[Lemma 4.37]{MV}.
\end{proof}
	
\subsection*{The iteration}
For a condition $\mathcal M_p\in\mathbb M^\kappa_\lambda$ and an ordinal $\gamma\in E$, we let 
$$\mathcal M_p(\gamma)\coloneqq\{M\in\mathcal M^{{\rm next}(\gamma)}_p:\gamma\in M\}.$$
In other words, $\mathcal M_p(\gamma)$ consists of those models in $\mathcal M_p{\rest{{\rm next}(\gamma)}}$ that are strongly active at ${\rm next}(\gamma)$.

Let $E^+= E \cup \{ \alpha + 1: \alpha \in E\}$. 
We will define by induction the poset $\mathbb Q^\kappa_\alpha$, for $\alpha \in E^+\cup \{ \lambda \}$. 
Let  ${\rm Fn}(\omega_1,\omega)$ denote the poset of finite partial functions from $\omega_1$  to $\omega$, ordered under reverse inclusion. 
Conditions in $\mathbb Q^\kappa_\alpha$ will be triples of the form $p= (\mathcal M_p,d_p,w_p)$, where $(\mathcal M_p,d_p) \in \mathbb P^\kappa_\alpha$,
and $w_p$ is a finite function with ${\rm dom}(w_p)\subseteq E\cap \alpha$, such that $w_p(\gamma)\in {\rm Fn}(\omega_1,\omega)$, for all $\gamma \in {\rm dom}(w_p)$. 
If $p$ is such a triple and $\gamma <\alpha$ is in $E$,  we let $p\rest \gamma$ denote the triple $(\mathcal M_p\rest \gamma, d_p \rest \gamma, w_p \rest \gamma)$, 
where  $(\mathcal M_p \rest \gamma, d_p\rest \gamma)$  is defined as in $\mathbb P^\kappa_\alpha$ and $w_p \rest \gamma$ is the restriction of $w_p$ to ${\rm dom}(w_p)\cap \gamma$.

	Let us recall that we are working with a suitable structure $(V_\lambda,\in,\kappa,U)$. To be  precise, $U:\lambda\rightarrow V_\lambda$ is a (bookkeeping) function that we regard it as a binary predicate. Thus let us assume that $U$ is onto and for every $x\in V_\lambda$, the set $\{\alpha<\lambda: U(\alpha)=x\}$ is unbounded.
\begin{definition}\label{iteration}
For $\alpha\in E$, we let  $\mathbb Q^\kappa_\alpha$ consist of triples $p=(\mathcal M_p,d_p,w_p)$, where
		
\begin{enumerate}
			
\item $(\mathcal M_p,d_p)\in\mathbb P^\kappa_\alpha$.
			
\item $w_p$ is a finite function  with  ${\rm dom}(w_p)\subseteq E\cap\alpha$ 
such that, if $\gamma\in {\rm dom}(w_p)$,  then  $U(\gamma)$
is a  $\mathbb Q^\kappa_\gamma$-term for  a relation on $\omega_1$
such that $\dot{T}_\gamma =(\check{\omega}_1,U(\gamma))$ is forced to be a tree of height $\omega_1$ without uncountable branches,
$w_p(\gamma)\in {\rm Fn}(\omega_1,\omega)$, and
\[ 
p \rest \gamma  \forces_{\mathbb Q^\kappa_\gamma} \check{w}_p(\gamma) \in \dot{\mathbb S}(\dot{T}_\gamma).
\]
				
\end{enumerate}
		
For conditions $p,q\in\mathbb Q^\kappa_\alpha$, we say $p$ is stronger than $q$ and write $p\leq q$, if and only if,
		
\begin{enumerate}
\item $(\mathcal M_p,d_p)\leq(\mathcal M_q,d_q)$ in $\mathbb P^\kappa_\alpha$,
\item ${\rm dom}(w_p)\supseteq{\rm dom}(w_q)$, and $w_q(\gamma)\subseteq w_p(\gamma)$,  for every $\gamma\in{\rm dom}(w_q)$. 

\end{enumerate}
We let $\mathbb Q^\kappa_{\alpha +1}$ denote the set of triples $p$ as above, but with  ${\rm dom}(w_p)\subseteq \alpha +1$. 
Let  $\mathbb Q^\kappa_\lambda=\bigcup_{\alpha\in E}\mathbb Q^\kappa_\alpha$ with the same order.
\end{definition}  
	
Note that if $\alpha = {\rm min}(E)$ then $\mathbb Q^\kappa_\alpha= \mathbb P^\kappa_\alpha$. 
Also, for $\alpha \in E$, $\mathbb Q^\kappa_{\alpha +1}$ is isomorphic to $\mathbb Q^\kappa_\alpha\ast \dot{\mathbb S}(\dot{T}_\alpha)$, if $U(\alpha)$
has the right form, otherwise $\mathbb Q^\kappa_{\alpha +1}\cong \mathbb Q^\kappa_\alpha$.
For $p \in \mathbb Q^\kappa_\lambda$ and $\alpha \in E$, we let 
\[
p \rest (\alpha +1) = (\mathcal M_p \rest \alpha, d_p\rest \alpha, w_p \rest (\alpha+1)).
\]
The order is transitive and whenever $p\in \mathbb Q^\kappa_\lambda$ and $\alpha \in E^+$,
then $p \rest \alpha \in \mathbb Q^\kappa_\alpha$. Moreover, if $p \leq q$ then $p \rest \alpha \leq q\rest \alpha$.

\begin{proposition}\label{subcomplete}
Suppose $p\in\mathbb Q^\kappa_\beta$, and that $\alpha\leq\beta$ are in  $E^+$. 
If $q$ is a condition in $\mathbb Q^\kappa_\alpha$ that extends $p\rest\alpha$,
then $p$ is compatible with $q$ in $\mathbb Q^\kappa_\beta$.
\end{proposition}
\begin{proof}
Let $(\mathcal M_r,d_r)$ be  $(\mathcal M_p,d_p)\land (\mathcal M_q,d_q) $ as defined in \cref{CS}. 
Let also $w_r$ be defined on 
${\rm dom}(w_p)\cup{\rm dom}(w_q)$ by
		
\begin{equation*}
w_r(\gamma)=\left\{ \begin{array}{cl}
w_q(\gamma) & \textrm{if }\gamma<\alpha,\\
w_p(\gamma) & \textrm{if }\gamma\geq\alpha.
\end{array}\right.
\end{equation*}

It is  evident that $r$ is a condition which  extends $p$ and $q$.
\end{proof}
\begin{remark} The condition $r$ from the previous lemma is the greatest lower bound of $p$
and $q$, so we will denote it by $r\coloneqq p\land q$.
\end{remark}
The following corollary is immediate.
\begin{corollary}
For every $\alpha\leq\beta$ in $E^+\cup\{\lambda\}$, $\mathbb Q^\kappa_\alpha$ is a complete suborder of $\mathbb Q^\kappa_\beta$.
\end{corollary}
\begin{proof}[\nopunct]
\end{proof}

\begin{proposition}\label{topMiteration}
Let $p\in \mathbb Q^{\kappa}_\lambda$ and let $M\in \mathscr C \cup \mathscr U$ be such that $p\in M$.
Then there is a  condition $p^M\leq p$ with $M\in \M_{p^M}$.
\end{proposition}
\begin{proof}
Let $\mathcal M_{p^M}$ and $d_{p^M}$ be defined as in \cref{M on top 1}. We let $p^M=(\mathcal M_{p^M},d_{p^M},w_p)$. 
It is clear that $p^M$ is the required condition. 
\end{proof}

 We need the following lemma in several proofs.
\begin{lemma}\label{amalgamation}
    Let $\alpha \in E$. Suppose $M\in\mathscr C\cup\mathscr U$ with $\eta(M)>\alpha$ and $\alpha \in M$. Let $p\in\mathbb Q^\kappa_\alpha$ be a condition 
with $M\rest \alpha \in \mathcal M_p$. Let $\beta=\max(\dom(w_p)\cap M)$. Let $G_{\beta+1}$ be a $V$-generic filter on $\mathbb Q^\kappa_{\beta+1}$ with $p\restriction{\beta+1}\in G_{\beta+1}$. Assume that $q\in M\cap\mathbb Q^\kappa_\alpha$ is a condition such that:
\begin{itemize}
\item $( \mathcal M_q,d_q)\leq (\mathcal M_p,d_p)\rest M$,
\item  $M\cap {\rm dom}(w_p)\subseteq{\rm dom}(w_q)$, and
\item $q\rest{\beta+1}\in G_{\beta+1}$.
\end{itemize}
Then $p$ and $q$ are compatible in $\mathbb Q^\kappa_\alpha$. 
\end{lemma}
\begin{proof}
    First, by \cref{SP},  we have that $(\mathcal M_q, d_q)$ and $(\mathcal M_p,d_p)$ are compatible in $\mathbb P^\kappa_\alpha$ and
the meet  $(\mathcal M_q, d_q) \wedge (\mathcal M_p,d_r)$ exists. Let us denote this meet by $(\mathcal M,d)$. 
Let us also fix $r\in G_{\beta+1}$  extending  $p\rest{\beta+1}$ and $q\rest{\beta+1}$. 
Let $(\mathcal M_s,d_s)$ be the meet $(\mathcal M,d) \wedge (\mathcal M_r,d_r)$ as defined in \cref{CS}. 
We now define $w_s$ on ${\rm dom}(w_r) \cup {\rm dom}(w_p) \cup {\rm dom}(w_q)$ by letting:	
\begin{equation*}
w_s(\gamma)=\begin{cases} 
w_r(\gamma) & \mbox{if } \gamma \in {\rm dom}(w_r) \\
w_q(\gamma) & \mbox{if }  \gamma\in {\rm dom}(w_q)\setminus {\rm dom}(w_r) ,\\
w_p(\gamma) & \mbox{if } \gamma \in {\rm dom}(w_p) \setminus ( {\rm dom}(w_r) \cup {\rm dom}(w_q) ).
\end{cases}
\end{equation*}
It is easy to see that $s=(\mathcal M_s,d_s,w_s)$ is a condition in $\mathbb Q^\kappa_\alpha$ extending $p$ and $q$.
\end{proof}
\begin{proposition}\label{proper-iteration}
Let $\alpha \in E$. Suppose $M\in\mathscr C\cup\mathscr U$ with $\eta(M)>\alpha$ and $\alpha \in M$. Let $p$ be a condition 
with $M\rest \alpha \in \mathcal M_p$. Then
\begin{enumerate} 
\item if $p\in \mathbb Q^\kappa_\alpha$, then $p$ is $(M,\mathbb Q^\kappa_\alpha)$-generic,
\item if $p\in \mathbb Q^\kappa_{\alpha+1}$,  then $p$ is $(M,\mathbb Q^\kappa_{\alpha+1})$-generic.
\end{enumerate} 
\end{proposition}
\begin{proof}
First note that under our assumptions both $\mathbb Q^\kappa_\alpha$ and $\mathbb Q^\kappa_{\alpha+1}$ belong to $M$. 
$\mathbb Q^\kappa_{\alpha+1}$ is either equal to $\mathbb Q^\kappa_\alpha$ or is isomorphic to $\mathbb Q^\kappa_\alpha \ast \dot{\mathbb S}(\dot{T}_\alpha)$.
Since $\dot{\mathbb S}(\dot{T}_\alpha)$ is forced to be ccc, (2) follows from (1). 
If $\alpha = {\rm min}(E)$ then $\mathbb Q^\kappa_\alpha$ is isomorphic to $\mathbb P^{\kappa}_\alpha$, and by \cref{SP}, it is then strongly proper for all models 
$M\in\mathscr C\cup\mathscr U$ with $\eta(M)>\alpha$.

Suppose now that $\alpha$ is not the least element of $E$ and  (2) holds for all $\beta < \alpha$. 
Let  $D\in M$ be a dense subset of  $\mathbb Q^\kappa_\alpha$.
We may assume without loss of generality that $p\in D$.
Let $\beta = {\rm max}({\rm dom}(w_p)\cap M)$. 
Pick a $V$-generic filter $G_{\beta+1}$ on $\mathbb Q^\kappa_{\beta+1}$ such that $p\rest{\beta+1}\in G_{\beta+1}$.
By elementarity of $M[G_{\beta+1}]$ in ${\rm Hull}(M,V_{\eta(M)})[G_{\beta+1}]$, there is a condition $q\in D\cap M[G_{\beta+1}]$ satisfying the  following.
\begin{enumerate}
\item $( \mathcal M_q,d_q)\leq (\mathcal M_p,d_p)\rest M$,
\item  $M\cap {\rm dom}(w_p)\subseteq{\rm dom}(w_q)$, and
\item $q\rest{\beta+1}\in G_{\beta+1}$.
\end{enumerate}
By the inductive assumption, $p \rest \beta+1$ is $(M,\mathbb Q^\kappa_{\beta+1})$-generic, and hence $M[G_{\beta+1}]\cap V=M$.
Therefore, we can find such $q$ in $M$. By \cref{amalgamation}, $p$ and $q$ are compatible.
\end{proof}

\begin{proposition}\label{proper-iteration2}
Suppose that $\lambda^*>\lambda$  is a sufficiently large regular cardinal. Let $p\in\mathbb Q^\kappa_\lambda$. Suppose that  $M^*\prec H_{\lambda^*}$ with $\kappa,\lambda\in M^*$ is such that $M\coloneqq M^*\cap V_\lambda \in\mathcal M_p$. Then $p$ is $(M^*,\mathbb Q^\kappa_\lambda)$-generic.
\end{proposition}
\begin{proof}
Observe that $\eta={\rm sup}(M^*\cap\lambda)$ is in $E$, and that  $M$ is an $\eta$-model. 
The rest is as in the proof of \cref{proper-iteration}.
\end{proof}
\begin{corollary}\label{properness whole iteration}
For every $\alpha\in E^+\cup\{\lambda\}$, $\mathbb Q^\kappa_\alpha$ is proper.
\end{corollary}
\begin{proof}[\nopunct]
\end{proof}

\begin{corollary}\label{U-strong}
For every $\alpha\in E^+\cup\{\lambda\}$, $\mathbb Q^\kappa_\alpha$ is strongly proper for $\mathscr U$.
\end{corollary}
\begin{proof}
This follows from \cref{proper-iteration,proper vs strong proper}.
\end{proof}

\begin{definition}\label{derived-sets} Let $\alpha \in E$ and let $G_\alpha$ be $V$-generic over $\mathbb Q^\kappa_\alpha$. 
In the model $V[G_\alpha]$, let $\mathscr C_{\rm st}[G_\alpha]$ denote the set of all $M\in \mathscr C_{\rm st}$ such that 
$\eta (M)>\alpha$, $\alpha \in M$ and $M\rest \alpha \in \mathcal M^\alpha_{G_\alpha}$. 
\end{definition} 
	
As in \cite[Lemma 5.2]{MV} we have the following. 
	
\begin{proposition}\label{derived-stationary} 
${\mathscr C}_{\rm st}[G_\alpha]$ is a stationary subset of ${\mathcal P}_{\omega_1}(V_\lambda)$ in the model $V[G_\alpha]$.
\end{proposition}
\begin{proof}[\nopunct]
\end{proof}

Let $\alpha\in E^+$. Assume that 	$G_\alpha$ is a $V$-generic filter on $\mathbb Q^\kappa_\alpha$. 
One can form the quotient forcing $\mathbb Q^\kappa_\lambda/G_\alpha$.
Suppose that also $M\in\mathscr C_{>\alpha}$, $M\rest \alpha \in \mathcal M^\alpha_{G_\alpha}$, and $p\in M\cap \mathbb Q^\kappa_\lambda/G_\alpha$. 
It is then easily seen that $p^M$ belongs to $\mathbb Q^\kappa_\lambda/G_\alpha$. 
By an argument, as in  \cref{proper-iteration},  we have the following.
	
\begin{proposition}\label{quotient-proper}
Let $\lambda^*>\lambda$ be a sufficiently large regular cardinal. Suppose that $\alpha\in E$, and $G_\alpha\subseteq \mathbb Q^\kappa_\alpha$ is a $V$-generic filter. 
Let $p \in \mathbb Q^\kappa_\lambda/G_\alpha$. Suppose $M^*\prec H_{\lambda^*}$ contains all the relevant objects, and $M= M^*\cap V_\lambda$ belongs to $\mathcal M_p$. 
Then $p$ is $(M^*[G_\alpha],\mathbb Q^\kappa_\lambda/G_\alpha)$-generic.
\end{proposition}
\begin{proof}[\nopunct]	\end{proof}
	
Suppose $\alpha \in E^+$ and let $G_\alpha$ be $V$-generic over $\mathbb Q^\kappa_\alpha$.  In $V[G_\alpha]$ fix a large regular cardinal $\lambda^*$ and let $\mathscr S$ be the collection of all countable models of the form $M[G_\alpha]$, where $M\prec H_{\lambda^*}^V$ contains the relevant objects, and $M \cap V_\lambda \in \mathscr C_{\rm st}[G_\alpha]$. Now, by  \cref{derived-stationary,quotient-proper} we have the following. 
	
\begin{corollary}\label{quotient-stat-proper}  $\mathscr S$ is stationary in ${\mathcal P}_{\omega_1}([H_{\lambda^*}[G_\alpha])$ and
$\mathbb Q^\kappa_\lambda/G_\alpha$ is $\mathscr S$-proper in $V[G_\lambda]$. 
\end{corollary}
\begin{proof}[\nopunct]	\end{proof}

\begin{proposition}\label{lambdacc-iteration}
$\mathbb Q^\kappa_\lambda$ satisfies the $\lambda$-c.c.
\end{proposition}
\begin{proof}
Assume that $A\subseteq \mathbb Q^{\kappa}_\lambda$ is of size $\lambda$. For every $p\in \mathbb{P}^\kappa_\lambda$, let $a(p)=\bigcup \{ a(M): M\in \M_p\}$, which is a closed subset of $E$ of size ${<}\kappa$. By a standard $\Delta$-system argument, we can find a subset $B$ of $A$ of size $\lambda$  so that there are $a$ and $d$ subsets of  $E$ such that $a(p)\cap a(q)= a$ and  ${\rm dom}(w_p)\cap{\rm dom}(w_q)=d$, for all distinct $p,q\in B$. Note that $a$ is closed, and if we let $\gamma = \max(a)$, then $\gamma \in E$. Since $B$ has size $\lambda$, by a simple counting argument, we can assume there is $\M\in \mathbb M^\kappa_\gamma$ such that $\M_p\rest \gamma = \M$ and  that  $w_p\rest d=w_q\rest d$, for all $p\in B$.  Now, pick distinct $p,q\in B$, and define $\M_r= \M_p \cup \M_q$, $d_r= d_p \cup d_q$, and also $w_r=w_p\cup w_q$. Let $r=(\M_r,d_r,w_r)$. It is straightforward to check that $r\in \mathbb Q^\kappa_\lambda$ and $r\leq p,q$. 
\end{proof}
Putting everything together, we have the following.
\begin{corollary}
$\mathbb Q^\kappa_\lambda$ preserves $\omega_1$, $\kappa$ and $\lambda$, and forces that $\kappa=\omega^{V[G_\lambda]}_2$ and $\lambda=\omega^{V[G_\lambda]}_3$.
\end{corollary}
\begin{proof}
The preservation of $\omega_1$ and $\kappa$ is guaranteed by \cref{properness whole iteration,U-strong}, respectively. It is easily seen that $\mathbb P^\kappa_\lambda$ is a complete suborder of $\mathbb Q^\kappa_\lambda$, and hence
\cref{collapse with pure side condition,lambdacc-iteration}  imply that in  generic extensions by $\mathbb Q^\kappa_\lambda$,
$\kappa=\omega_2$ and $\lambda=\omega_3$.
\end{proof}

\begin{proposition}\label{approx}
For every $\alpha\in E^+\cup\{\lambda\}$, $\mathbb Q^\kappa_\alpha$ has the $\omega_1$-approximation property.
\end{proposition}
\begin{proof}
We proceed by induction. Suppose $\alpha\in E$ and we have established that $\mathbb Q^\kappa_\alpha$
has the $\omega_1$-approximation property.  Recall that $\mathbb Q^\kappa_{\alpha+1}$ is either $\mathbb Q^\kappa_\alpha$
or is isomorphic to $\mathbb Q^\kappa_\alpha \ast \dot{\mathbb S}(\dot{T}_\alpha)$, for some $\dot{T}_\alpha$
which is a name for a tree of size and height $\omega_1$ without uncountable branches. 
Then by  \cref{spec-poset-approx}  $\dot{\mathbb S}(\dot{T}_\alpha)$
is forced to have the $\omega_1$-approximation property over $V[G_\alpha]$, where $G_\alpha$ is a generic over ${\mathbb Q^\kappa_\alpha}$.
By  \cref{transitivity of approx} we conclude that $\mathbb Q^\kappa_{\alpha+1}$ 
has the $\omega_1$-approximation property. 
Recall also that $\mathbb Q^\kappa_{{\rm min}(E)}$ is strongly proper  for $\mathscr C$, by \cref{SP},
and hence has the $\omega_1$-approximation property by \cref{SP implies approx}.

Suppose now that $\alpha$ is not the least element of $E$ and that the conclusion holds for every ordinal in  $E^+\cap \alpha$. 
Let $\dot f$ be a $\mathbb Q^\kappa_\alpha$-name for a function $\mu \to 2$ which is forced by a condition $p$ to be 
countably approximated in $V$. We may assume that  $\mu$ is a cardinal in $V$.
Suppose $\lambda^*>\mu,\lambda$ is a sufficiently large regular cardinal.
Pick a countable  $M^*\prec H_{\lambda^*}$ containing the  relevant objects. Thus $M=M^*\cap V_\lambda$ is a standard virtual model.  
Let $q\leq p^M$ be a condition which decides $\dot f\rest {M^*\cap \mu}$ and forces it to be equal to some function  $g:M^*\cap\mu\rightarrow 2$ 
which is in $V$. 
By \cref{proper-iteration2}  $q$ is $(M^*,\mathbb Q^\kappa_\alpha)$-generic.
By \cref{decision is guessed}, it suffices to show that $g$ is guessed in $M^*$. Assume towards a contradiction that $g$ is not guessed in $M^*$.
Let $\gamma={\rm max}({\rm dom}(w_q)\cap M)$, 
and pick some $V$-generic filter $G_{\gamma+1}$ on $\mathbb Q^\kappa_{\gamma+1}$
such that 	$(q\rest  \gamma,w_q(\gamma))\in G_{\gamma+1}$.
Working in $V[G_{\gamma+1}]$, we show that in  $M^*[G_{\gamma+1}]$ there is an assignment $x\mapsto (q_x,g_x)$ on $[\mu]^\omega\cap V$ 
with the following properties.
			
\begin{enumerate}
\item $q_x\in\mathbb Q^\kappa_\alpha$,
\item $(\mathcal M_{q_x},d_{q_x})\leq (\mathcal M_q,d_q)\rest M$,
\item $q_x\rest{\gamma+1}\in G_{\gamma+1}$,
\item $g_x: x\to 2$, $g_x \in V$,
\item $q_x\Vdash \dot{f}\rest x=\check{g}_x$.
\end{enumerate}
			
If $x\in [\mu]^\omega \cap M^*$ then $x\mapsto (q,g\rest x)$ satisfies all the above properties.
By  elementarity of $M^*[G_{\gamma+1}]$ in $H_{\lambda^*}[G_{\gamma+1}]$,  such an assignment exists for all $x \in [\mu]^\omega \cap V$.
Moreover, by elementarity again there is such an assignment in $M^*[G_{\gamma+1}]$.
Now, by the inductive assumption and \cref{approx does not add guessing function}, $g$ is not guessed in $M^*[G_{\gamma+1}]$.
Since $\mathbb Q^\kappa_\alpha$ is proper, $[\mu]^\omega \cap V$ is a cofinal in $[\mu]^\omega$.
By  \cref{Baumgartner} there is $B\in M^*[G_{\gamma+1}]$, a  cofinal subset of $[\mu]^\omega\cap V$, 
such that for every $x\in B\cap M^*[G_{\gamma+1}]$, $g_x\nsubseteq g$, and hence $q_x$ is incompatible with $q$.
On the other hand, conditions (1)-(3) above enable us to use \cref{amalgamation} to make sure that $q_x$ and $q$ are compatible. 
Thus, for every such $x\in B\cap M^*[G_{\gamma+1}]$, $q_x$ and $q$ are compatible, and hence we get a contradiction. 
\end{proof}
One can  use \cref{quotient-proper} and the proof of \cref{approx} to prove the following.
\begin{proposition}\label{quotient-approximation-5.35}
Suppose $\alpha, \beta \in E^+ \cup \{ \lambda \}$ and $\alpha < \beta$. 
Suppose that $G_\alpha$ is a $V$-generic filter over $\mathbb Q^\kappa_\alpha$. 
Then $\mathbb Q^\kappa_\beta/G_\alpha$ has the $\omega_1$-approximation property over $V[G_\alpha]$.
\end{proposition}
\begin{proof}[\nopunct]
\end{proof}
	
The following corollary is immediate.
\begin{corollary}\label{level by level approx}
Suppose $\alpha, \beta \in E^+ \cup \{ \lambda \}$  and $\alpha <\beta$. Let $G_\beta$ be a $V$-generic filter over $\mathbb Q^\kappa_\beta$,
and let $G_\alpha = G_\beta \cap \mathbb Q^\kappa_\alpha$. Then the pairs $(V,V[G_\alpha])$ and  $(V[G_\alpha],V[G_\beta ])$
have the $\omega_1$-approximation property.
\end{corollary}
	
\begin{proof}[\nopunct]			\end{proof}

\subsection*{Quotients by Magidor models}

The first lemma of this subsection states that 
$\mathbb Q^\kappa_\lambda$ is strongly proper for  $\mathscr U$ in a canonical way.
	
\begin{definition}
Let $p\in\mathbb Q^\kappa_\alpha$. Assume  $N\in \mathcal L(\mathcal M_p)$ is a Magidor model. 
We let 
\[
p\rest N=(\mathcal M_{p\rest N},d_{p\rest N},w_p\rest N).
\]
\end{definition}

\begin{lemma}\label{two projections commute} 
Let  $p\in\mathbb Q^\kappa_\lambda$, and let   $N\in\mathcal L(\mathcal{M}_p)$ be a Magidor model. Assume that $\alpha\leq \eta(N)$ is in $E$.
Then $p\rest \alpha\rest N=(p\rest N)\rest  \alpha$.
\end{lemma}
\begin{proof}
It is enough to show that if $M\in\mathcal M_p$, then $(M\rest N)\rest  \alpha= M\rest  \alpha\rest N$.
This is of course clear from the definition of $M\rest N$.
\end{proof}

\begin{lemma}\label{Canonical Uproper-iteration}
    Let $p\in\mathbb Q^\kappa_\alpha$, and assume that  $N\in \mathcal L(\mathcal M_p)$ is a Magidor model.   Then
    \begin{enumerate}
        \item $p\rest N\in \mathbb Q^\kappa_\alpha\cap N$.
        \item Every condition $q\in N$  extending $ p\rest N$  is compatible with $p$.
    \end{enumerate} 
\end{lemma}

\begin{proof}
We prove both items simultaneously by induction on $\alpha$.
 Note that both are true for the minimum of $E$. Thus let us assume that $\alpha>\min(E)$.
 Note that  $p\rest N$ belongs to $N$. So all we have  to show is that $p\rest N$ is a condition in $\mathbb Q^\kappa_\alpha$ that satisfies the second item.

One needs first to show that for every $\gamma\in\dom(p)\cap N$, $(p\rest N)\rest\gamma$, which is a condition by the inductive hypothesis and \cref{two projections commute}, forces in  ${\mathbb Q^\kappa_\gamma}$ that $``\check{w}_p(\gamma) \in \dot{\mathbb S}(\dot{T}_\gamma)"$. 
Assume towards a contradiction that this is not the case, so there is $q\in {\mathbb Q^\kappa_\gamma}$ with $q\leq (p\rest N)\rest\gamma$ which forces 
$\check{w}_p(\gamma)$ is not in $\dot{\mathbb S}(\dot{T}_\gamma)"$. Since  both $p\rest N\rest \gamma$ and $w_p(\gamma)$ belong to $N$, we can find such $q$ in $N$ by elementarity. Therefore, $q$ is not compatible with $p\rest\gamma$. This contradicts the second item above applied  to $p\rest\gamma,N\rest\gamma$, and $q$. Hence the first item follows.

To see the second item, we will define a common extension of $p,q$, say $r$ by letting $(\mathcal M_r,d_r)=(\mathcal M_p,d_p)\land(\mathcal M_q,d_q)$,  $\dom(w_r)={\rm dom}(w_p)\cup{\rm dom}(w_q)$, and that 
\begin{equation*}
w_r(\gamma)=\begin{cases} 
w_q(\gamma) & \mbox{if } \gamma\in N, \\
w_p(\gamma) & \mbox{if } \gamma\notin N.
\end{cases}
\end{equation*}

It follows from  \cref{rest-magidor-prop}  that $r$ is a common extension of $p$ and $q$.
\end{proof}

\begin{remark}
As before, we let the above $r$  be the meet of $p$ and $q$, which we denote by $p\land q$.
\end{remark}
Suppose that $N \in\mathscr U$ and $\alpha\in E$. We set
\[	\mathbb Q^\kappa_\alpha\rest N\coloneqq\{p\in\mathbb Q^\kappa_\alpha: N\rest  \alpha\in\mathcal M_p\}~~~\mbox{ and }~~~\mathbb Q^\kappa_{\alpha,N}\coloneqq\mathbb Q^\kappa_\alpha\cap N.
\]

By \cref{Canonical Uproper-iteration}, the condition $\mathbf 1_{N}\coloneqq(\{N\rest  \alpha\},\varnothing,\varnothing)$ is $(N,\mathbb Q^\kappa_\alpha)$-strongly generic, and therefore the mapping
\begin{equation*}
\begin{split}
\mathbb Q^\kappa_{\alpha,N}&\longrightarrow \mathbb Q^\kappa_\alpha\rest N\\
p&\longmapsto p^{N\rest \alpha}
\end{split}
\end{equation*}
is a complete embedding. 
Moreover, if $p$ is a condition in $\mathbb Q^\kappa_\alpha\rest N$, then $p\rest N\in \mathbb Q^\kappa_{\alpha,N}$ is such that if $q\in \mathbb Q^\kappa_{\alpha,N}$ extends $p\rest N$, then $p$ and $q$ are compatible, and indeed the meet $p\land q$ exists.

For a $V$-generic filter $G_{\alpha,N}\subseteq\mathbb Q^\kappa_{\alpha}\cap N$ we can form the following quotient forcing
\[
\mathbb R^N_\alpha\coloneqq\mathbb Q^\kappa_\alpha\rest N/G_{\alpha,N}.
\]
\begin{definition}
Let $G_{\alpha,N}$ be a $V$-generic filter over $\mathbb Q^\kappa_{\alpha}\cap N$. 
Let $\mathscr C_{\rm st}[G_{\alpha,N}]$ denote the set of all models $M\in \mathscr C_{\rm st}$ such 
that $\alpha, N\in M$ and $(N\land M)\rest{\alpha}\in \M_{G_{\alpha,N}}$.
\end{definition}
	
Note that $\mathscr C_{\rm st}[G_{\alpha,N}]$ is stationary in $\mathcal P_{\omega_1}(V_\lambda)$ in the model $V[G_{\alpha,N}]$, but we do not need this and therefore avoid a proof.

\begin{lemma}[Factorization Lemma]\label{project to quotient}
Suppose $N\in\mathscr U$, and that $\alpha\leq\beta\leq\eta(N)$ are in $E$. Let $G_{\alpha,N}\subseteq \mathbb Q^\kappa_\alpha \cap N$ be a $V$-generic filter. The mapping 
\begin{equation*}
\begin{split}
\rho:\mathbb Q^\kappa_\beta\rest N/G_{\alpha,N}&\longrightarrow \mathbb R^N_\alpha \times (\mathbb Q^\kappa_\beta\cap N)/G_{\alpha,N}\\
p&\longmapsto (p\rest  \alpha, p\rest N)
\end{split}
\end{equation*}
is a projection in $V[G_{\alpha,N}]$.
\end{lemma}
\begin{proof}
It is easy to observe that $\rho$  is a well-defined, order-preserving mapping that respects the maximal conditions.
Assume that the arbitrary  elements
\[p\in\mathbb Q^\kappa_\beta\rest N/G_{\alpha,N} \mbox{ and } (r,s)\in\mathbb Q^\kappa_\alpha\rest{N}/G_{\alpha, N}  \times (\mathbb Q^\kappa_\beta\cap N)/G_{\alpha,N} \mbox{ with }\]
\[
(r,s)\leq (p\rest  \alpha,p\rest N)
\]
are given. Our goal is to find a condition $q\in\mathbb Q^\kappa_\beta\rest N/G_{\alpha,N}$ with $q\leq p$ such that $\rho(q)\leq (r,s)$.

Since  $r\rest N,s\rest  \alpha\in G_{\alpha,N}$, we can fix a common extension $\bar q\in G_{\alpha,N}$ of them. Now by applying \cref{Canonical Uproper-iteration} to $\bar{q}, N,r$, we have 
\[
\bar q\land r\leq \bar q\leq s\rest\alpha,
\]
and by applying \cref{subcomplete} to $\bar{q},s,\alpha$, we have
\[
\bar q\land s\leq \bar q\leq r\rest N.
\]
By reverse applications of \cref{Canonical Uproper-iteration,subcomplete} to the above inequalities, the meets $(\bar q\land r)\land s$ and $(\bar q\land s)\land r$ exist. A simple calculation shows that these conditions are equal, so let us call them
$q^*$. The canonicity of our projections guarantees that 
\[
q^*\rest  \alpha=(\bar q\land r)\leq r~~~\mbox{ and }~~~q^*\rest N=(\bar q\land s)\leq s. 
\]
Moreover,
$$q^*\rest  \alpha\rest N= (q^*\rest N)\rest  \alpha=\bar q\in G_{\alpha,N}.$$	
Consequently,
\[
q^*\in\mathbb  Q^\kappa_\beta/G_{\alpha,N}~~~\mbox{ and }~~~(q^*\rest  \alpha,q^*\rest N)\leq (r,s).
\]
It suffices to show that $p$ and $q^*$ are compatible, since the mappings $\bullet\mapsto \bullet\rest\alpha$ and $\bullet\mapsto\bullet\rest N$ are order-preserving.
We define a common extension  of them, say $q$. Let $\mathcal M_q=\mathcal M_p\cup\mathcal M_{q^*}$. For every $\delta\in E$,
\begin{equation*}
\mathcal M_q^\delta=\begin{cases} 
\mathcal M_{q^*}^\delta & \mbox{if } \delta\leq\alpha,~~~~~\mbox{ (since $\mathcal M_p^\delta\subseteq\mathcal M_r^\delta\subseteq\mathcal M_{q^*}^\delta$)} \\
& \\
\mathcal M_p^\delta & \mbox{if } \delta>\alpha.~~~~~~\mbox{ (since $\mathcal M_{q^*}^\delta\subseteq\mathcal M_s^\delta\subseteq\mathcal M_p^\delta$)} 
\end{cases}
\end{equation*}
In either case, $\mathcal M_q^\delta$ is  an $\delta$-chain. Thus $\mathcal M_q$ is a condition in $\mathbb M^\kappa_\alpha$.
We now define $d_q$  on ${\rm dom}(d_p)\cup {\rm dom}(d_{q^*})$ by letting
		
\begin{equation*}
d_q(P) = \begin{cases} 
d_p(P) &\mbox{if } \eta(P)>\alpha \text{ and } P\notin N, \\
d_{q^*}(P) & \mbox{otherwise}. 
\end{cases}
\end{equation*}
The function is well-defined as $\rho(q^*)\leq \rho(p)$.		
A proof similar  to \cite[Lemma 4.14]{MV} would show that every model in ${\rm dom}(d_q)$ is $\mathcal M_q$-free.
We sketch the proof. Thus we 
 assume  that $P\in\mathcal L(\mathcal M_q)$, $Q\in\mathcal M_q$, and that $P\in_{\eta(P)}$ Q. We have to show that $Q$ is strongly active at $\eta(P)$.
\begin{itemize}
    \item \textbf{Case 1.}  $\eta(N)>\alpha$ and $P\notin N$:
    Then  $P$ does not belong to $\mathcal L(\mathcal M_{q^*})$, hence $Q\notin\mathcal M_{q^*}$, in which case, since  $p$ is a condition, $Q$ must be strongly active at $\eta(P)$.
     \item \textbf{Case 2.} $\eta(N)>\alpha$ and  $P\in N$:
    Then  $P$ is in $\mathcal L(\mathcal M_{q^*})$. We may assume that $Q\notin\mathcal M_{q^*}$. In particular, $Q\in\mathcal M_p$. Now both $N$ and $Q$ are active at $\eta(P)$. Therefore, we must have $P\in N\in_{\eta(P)} Q$.
    We can also assume that $Q$ is countable, as otherwise, it is strongly active at $\eta(P)$.
    Now, we have $P\in_{\eta(P)}(N\land Q)\rest N$, where the latter model belongs to $\mathcal M_{q^*}$. Thus $(N\land Q)\rest N$ is strongly active at $\eta(P)$, and hence $Q$.

    \item \textbf{Case 3.}  $\eta(P)\leq\alpha$: In this case $P$ is in $\mathcal L(\mathcal M_{q^*\rest  \alpha})$, then $Q\rest  \alpha\in \mathcal M_{q^*\rest\alpha}$.
Now $Q\rest  \alpha$, and consequently $Q$ is strongly active at $\eta(P)$.	
\end{itemize}

Using the above argument, it is easy to show that the
 decorative function $d_q$ fulfils 
$(\divideontimes)$ in \cref{PF}.

It remains to define $w_q$ on ${\rm dom}(w_p)\cup{\rm dom}(w_{q^*})$. For each $\gamma\in {\rm dom}(w_q)$, we let
\begin{equation*}
w_q(\gamma)=\left\{ \begin{array}{cl}
w_p(\gamma) & \textrm{if }\gamma>\alpha \text{ and } \gamma\notin N, \\
w_{q^*}(\gamma) & \textrm{otherwise}.
\end{array}\right.
\end{equation*}
		
We observe that $w_q$ is well-defined thanks to the definition of $q^*$. 
Thus $q$ is a condition in $\mathbb Q^\kappa_\beta\rest N$, which is easily shown to be a common extension of  $p$ and $q$.
By an easy calculation, we have  \[
q\rest \alpha\rest N=\bar q\in G_{\alpha,N}, 
\]
and therefore $q\in\mathbb Q^\kappa_\beta\rest N/G_{\alpha,N}$, as required!
\end{proof}

\begin{lemma}\label{subcomplete R}
Suppose  $N\in\mathscr U$ and that $\alpha\leq\beta\leq \eta(N)$ are in $E$. Let $G_{\beta,N}$ be a $V$-generic filter over $\mathbb Q^\kappa_{\beta,N}$. 
Then in $V[G_{\beta,N}]$, $\mathbb R^N_\alpha$ is a complete suborder of $\mathbb R^N_\beta$.
\end{lemma}
\begin{proof}
It is clear that $\mathbb R^N_\alpha\subseteq\mathbb R^N_\beta$, and furthermore, if $p\in\mathbb R^N_\beta$, then $p\rest\alpha\in\mathbb R^N_\alpha$. 
Now fix $p\in\mathbb R^N_\beta$ and suppose that $q\in\mathbb R^N_\alpha$ extends $p\rest \alpha$. Let $r=p\land q$ as obtained in $\mathbb Q^\kappa_\beta$.
Therefore, we have $r\in\mathbb Q^\kappa_\beta\rest N$. An easy calculation shows that  \[
r{\rest N}=p\rest N\land q\rest N\in G_{\beta,N},
\]
where we observe that $p\rest N\leq (q\rest N)\rest\alpha$, and hence the meet $p\rest N\land q\rest N$ exists by \cref{subcomplete}.
\end{proof}

\begin{lemma}\label{M on top in Q by Magidor}
Assume  $N$ is a Magidor model, $M\in\mathscr C_{\rm st}[G_{\alpha,N}]$, and that $p\in\mathbb R^N_\alpha\cap M$. Then $p^{M\rest  \alpha}$ is an
$(M[G_{\alpha,N}],\mathbb R^N_\alpha)$-generic condition.
\end{lemma}
\begin{proof}
Let us first show that $p^{M\rest\alpha}$ is an $\mathbb R^N_\alpha$-condition.
Note that by the proof of \cite[Lemma 5.7 ]{MV}, we have
\[
(\mathcal M_p,d_p)^{M\rest\alpha}\rest{N}=\big((\mathcal M_p,d_p)\rest{N}\big)^{N\land M\rest\alpha}.
\] 
On the other hand,
$N\land M\rest\alpha$ is in $\mathcal M_{G_{\alpha,N}}$, which in turn implies that
\[
\big((\mathcal M_{p\rest{N}},
d_{p\rest{N}})^{N\land M\rest\alpha},\varnothing\big)\in G_{\alpha,N}.
\]

It follows that
\[
p^{M\rest\alpha}\rest{N}=\big((\mathcal M_{p\rest{N}},
d_{p\rest{N}})^{N\land M\rest\alpha},w_p\rest{N}\big)\in G_{\alpha,N}.
\]

This shows that $p^{M\rest\alpha}$ is a condition in $\mathbb R^N_\alpha$, and that
$p^{M\rest\alpha}$ is $(M,\mathbb Q^\kappa_\alpha\rest  N)$-generic. Therefore, $p^{M\rest\alpha}\rest  N$ is $(M,\mathbb Q^\kappa_\alpha\cap N)$-generic, and
that 
\[
p^{M\rest\alpha}\rest{N}\Vdash``\check{p} \mbox{ is } (M[\dot{G}_{\alpha,N}],\dot{\mathbb R}^N_\alpha)\mbox{-generic}".
\]
Since $p^{M\rest\alpha}\rest  N$ belongs to $G_{\alpha, N}$, we have, in $V[G_{\alpha,N}]$,  $p$ is $(M[G_{\alpha,N}],\mathbb R^N_\alpha)$-generic.
\end{proof}

We now state our key lemma.
	
\begin{lemma}\label{Key lemma in app}
Suppose  $N\in\mathscr U$ and that $\gamma\in a(N)$.
Let $G_{\gamma, N}$ be a $V$-generic filter over $\mathbb Q^\kappa_\gamma\cap N$.
Then in $V[G_{\gamma, N}]$,  $\mathbb R^N_\gamma$ has the $\omega_1$-approximation property.
\end{lemma}
\begin{proof}
We will prove by  induction on $\beta\leq \gamma$ that  $\mathbb R^N_\beta$ has the $\omega_1$-approximation property
over $V[G_{\gamma, N}]$, where in the definition of $\mathbb R^N_\beta$, we use $G_{\gamma,N}\cap \mathbb Q^\kappa_\beta$ as the generic filter on $\mathbb Q^\kappa_\beta\cap N$.
Thus we fix $\beta$ and assume that  the lemma holds for every ordinal in $E\cap\beta$.

Suppose that $\dot f$ is an $\mathbb R^N_\beta$-term forced by a condition $p$ to be a function on an ordinal $\mu$  that is  countably approximated  in $V[G_{\gamma, N}]$. We look for an exetnsion of $p$ that decides $\dot{f}$. We first choose a countable model  $M$ elementary in  $V_\delta$ with $\delta>\lambda$ such that $E,\beta,\gamma, N, p, \mu$ and $\dot f$ belong to $M$, and that
\[
\big(N\land (M\cap V_\lambda)\big)\rest  \gamma\in \mathcal M_{G_{\gamma, N}}.
\]

Let also
\[
M_\gamma\coloneqq(M\cap V_\lambda)\rest \gamma~~~\mbox{ and }~~~M_\beta\coloneqq M_\gamma\rest \beta.
\]
Using \cref{M on top in Q by Magidor}, we can  extend $p$ to a condition  $p^{M_\beta}$  in $\mathbb R^N_\beta$  so that $M_\beta\in\mathcal M_{p^{M_\beta}}$. Since $p$ forces that $\dot f$ is  $\omega_1$-approximated in $V[G_{{\gamma},N}]$, there is a condition
$q\in\mathbb R^N_\beta$ with  $q\leq p^{M_\beta}$ which decides the values of $\dot f\rest  M\cap \mu$.
Therefore, there exists a fucntion $g:M\cap\mu\rightarrow 2$~ such that
$$q\Vdash ``\check{g}=\dot f\rest  M\cap \mu".$$

 \cref{subcomplete R,M on top in Q by Magidor} ensures that $p^{M_\beta}$  is a$(M_\gamma[G_{\gamma,N}],\mathbb R^N_\beta)$-generic condition, and thus $(M[G_{\gamma,N}],\mathbb R^N_\beta)$-generic. On the other hand, by \cref{decision is guessed}, we are done if $g$ is guessed in $M[G_{\gamma, N}]$. 
Therefore, we may assume towards a contradiction that $g$ is not guessed in $M[G_{\gamma, N}]$, and we let 
\[
\alpha\coloneqq{\rm max}({\rm dom}(w_q)\cap M).
\] 
Assume that  $G^N_\alpha$ is a $V[G_{\gamma,N}]$-generic filter on $ \mathbb R^N_\alpha$ containing $q\rest  \alpha$, where
again we use \[
G_{\alpha,N}=G_{\gamma,N}\cap\mathbb Q^\kappa_\alpha
\] 
to form the quotient forcing $\mathbb R^N_\alpha$.
Notice that	\cref{project to quotient,M on top in Q by Magidor} imply that $q\rest  \alpha$  is $(M[G_{\gamma,N}],\mathbb R^N_\alpha)$-generic, and on the other hand  by the inductive hypothesis, $\mathbb R^N_\alpha$  has the $\omega_1$-approximation property over $V[G_{\gamma,N}]$. Consequently,  the function $g$ is not  guessed in $M[G_{\gamma,N}][G^N_\alpha]$ by \cref{Baumgartner}. 
Since $G^N_\alpha$ is a $V[G_{\gamma,N}]$-generic filter on $\mathbb R^N_\alpha$, the factorization lemma (\cref{project to quotient}) over $V[G_{\alpha,N}]$ yields the following equality:
\[
V[G_{\alpha,N}][G_{\gamma, N}][G^N_\alpha]=V[G_{\alpha,N}][G^N_\alpha][G_{\gamma,N}].
\]
 Note that $G^N_\alpha$ is a $V$-generic filter on $\mathbb Q^\kappa_\alpha\restriction N$ and $G_{\gamma,N}$ is a $V$-generic filter on $\mathbb Q^\kappa_\gamma\cap N$, and hence we have
\[
V[G_{\gamma, N}][G^N_\alpha]=V[G^N_\alpha][G_{\gamma,N}].
\]

\begin{claim}\label{claim11}
The pair $(V[G^N_\alpha],V[G^N_\alpha][G_{\gamma,N}])$ has the $\omega_1$-approximation property.
\end{claim} 
\begin{proof}
Let $G^N_\gamma$ be a $V[G^N_\alpha][G_{\gamma,N}]$-generic filter on $\mathbb R^N_\gamma$. Notice that both
$G^N_\alpha$ and $G^N_\gamma$ are also  $V$-generic filters on $\mathbb Q^\kappa_\gamma$ and $\mathbb Q^\kappa_\alpha$, respectively.
Now, \cref{level by level approx} implies that the pair 
\[(V[G^N_\alpha],V[G^N_\gamma])
\]
has the $\omega_1$-approximation property, 
and so does  the pair $(V[G^N_\alpha],V[G^N_\alpha][G_{\gamma,N}])$.
\end{proof}
		
Note that if $\alpha\in N$, then $\dot{\mathbb S}(\dot{T}_\alpha)$ is already interpreted by $G_{\alpha,N}$, and hence by $G_{\gamma,N}$ and moreover there is a generic filter on ${\rm val}_{G_{\gamma,N}}(\dot{\mathbb S}(\dot{T}_\alpha))$ in $V[G_{\gamma,N}]$. Thus in this case, the conditions in $\dot{\mathbb S}(\dot{T}_\alpha)$ cannot prevent us from amalgamating conditions in $\mathbb R^N_\beta$ (i.e., if $p'$ and $q'$ are conditions in $\mathbb R^N_\beta$ with $\alpha\in {\rm dom}(w_{p'})\cap{\rm dom}(w_{q'})$, we do know that $w^{G_{\gamma, N}}_{p'}(\alpha)$ and $w^{G_{\gamma,N}}_{q'}(\alpha)$ are compatible.)

For the sake of simplicity, we make the following convention.
\begin{convention}
We let $\mathbb Q$ be the interpretation of $\dot{\mathbb S}(\dot{T}_\alpha)$ under $G^N_\alpha$  if $\alpha$ is not in $N$, and  let it be the trivial forcing otherwise. We also let,  for every $p'\in\mathbb R^N_\beta$, $z_{p'}=w^{G^N_\alpha}_{p'}(\alpha)$ if $\alpha\notin N$,  and  let it be a canonical $\mathbb R^N_\alpha$-name in $N$ for the  maximal condition of the trivial forcing otherwise.
\end{convention}

\begin{claim}\label{claim12}
The forcing $\mathbb Q$ has the $\omega_1$-approximation property over
$V[G_{\gamma,N}][G^N_\alpha]$, and  $z_q$ is 
$(M[G_{\gamma,N}][G^N_\alpha],\mathbb Q)$-generic.
\end{claim}
\begin{proof}
We may assume that $\alpha\notin N$ and that $\mathbb Q$ is nontrivial; thus $\mathbb Q$ is the specializing forcing  of a tree $T\in V[G^N_\alpha]$ of size and height $\omega_1$ without any cofinal branches. By \cref{claim11}, the tree $T$ is still of height $\omega_1$ and has no cofinal branches in $V[G^N_\alpha][G_{\gamma,N}]$.  On the other hand,     $V[G^N_\alpha][G_{\gamma,N}]=V[G_{\gamma,N}] [G^N_\alpha]$.
Thus  in $V[G_{\gamma,N}][G^N_\alpha]$,    $\mathbb Q$ is
a c.c.c. forcing with the $\omega_1$-approximation property. Therefore, 
$z_q$ is  $(M[G_{\gamma,N}][G^N_\alpha],\mathbb Q)$-generic.
\end{proof}
		
Returning to the main body of the proof, let $H$ be a $V[G_{\gamma,N}][G^N_\alpha]$-generic filter on $\mathbb Q$ containing $z_q$. We now work in the model $V[G_{\gamma,N}][G^N_\alpha][H]$ for the rest of the proof. By  \cref{claim12,approx does not add guessing function} and the fact that we have assumed earlier  that $g$
 is not guessed in $M[G_{\gamma,N}][G^N_\alpha]$, we have that $g$ is not guessed in 
\[
M^*\coloneqq M[G_{\gamma,N}][G^N_\alpha][H].
\]

Note that for every $x\in M^*\cap [\mu]^\omega$,  $x\mapsto (q,g)$ witnesses the  conditions below, and so  a mapping $x\mapsto (q_x,g_x)$ on $[\mu]^{\omega}$ exists in $M^*$ 
such that:  
		
\begin{enumerate}
\item $q_x\in \mathbb Q^\kappa_\beta\rest  N$,
\item $(\mathcal M_{q_x},d_{q_x})\leq(\mathcal M_p,d_p)\rest  M$ in $\mathbb P^\kappa_\beta$,
\item $M\cap {\rm dom}(w_q)\subseteq{\rm dom}(w_{q_x})$,
\item $q_x\rest  \alpha\in G^N_\alpha$,
\item $z_{q_x}\in H$
\item $q_x\rest  N\in G_{ \beta,N}$, and
\item $g_x$ is a function with $x\subseteq {\rm dom}(g_x)$ such that $q_x\Vdash ``g_x\rest  x=\dot{f}\rest  x"$.

\end{enumerate}

Since  $g$ is not guessed in $M^*$,
by \cref{Baumgartner},  there is a set $B\in M^*$ cofinal in $[\mu]^\omega$ such that for every  $x\in B\cap M^*$, $g_x\nsubseteq g$. Therefore, the conditions $q_x$ and $q$  are incompatible in $\mathbb R^\beta_N$, for every $x\in B\cap M^*$. 

We now draw a contradiction using the following claim. 
\begin{claim}
For every $x\in B\cap M^*$, $q_x$ and $q$ are compatible. 
\end{claim}
\begin{proof}
Fix $x\in B\cap M^*$. Thus $q_x\in M^*$, and hence 
 $q_x\in M$ by \cref{project to quotient,M on top in Q by Magidor,claim12}. The conditions (1)-(5) above allow us to use \cref{amalgamation} to find a common extension of $q_x$ and $q$ in $\mathbb Q^\kappa_\beta\rest  N$, say $\bar{q}_0$. However, we need to show that $q_x$ and $q$ are compatible in $\mathbb R^N_\beta$.
Let   $r\in G^N_\alpha$ be a common extension of $q_x\rest  \alpha$ and $q\rest  \alpha$, and let
also $s\in G_{ \beta,N}$ be a common extension of $q_x\rest  N$ and $q\rest  N$.
 Set
$$d=\{\xi\in {\rm dom}(w_{\bar{q}_0}): \alpha\leq \xi\notin N\},$$
		
and let
\[
\bar{q}_1\coloneqq(\mathcal M_{\bar{q}_0},d_{\bar{q}_0}, w_{\bar{q}_0}\rest  d).
\] 
Notice that $\bar{q}_1\in\mathbb Q^\kappa_\beta\rest  N$, and that 
\[
(r,s)\leq (\bar{q}_1\rest  \alpha,\bar{q}_1\rest  N).
\]
We now apply \cref{project to quotient} to $(r,s)$ and $\bar{q}_1$
to  find a condition $\bar{q}\in\mathbb Q^\kappa_\beta\rest  N/G_{\alpha,N}$
 with $\bar{q}\leq r,s,\bar{q}_1$ such that $\bar{q}\rest  N\in G_{\beta,N}$. 
Note that $\bar{q}\leq q_x,q$ in $\mathbb Q^\kappa_\beta\rest  N$ but $\bar{q}$ is in $\mathbb R^N_\beta$, and hence $q$ and $q_x$
are compatible in $\mathbb R^N_\beta$, as required.
\end{proof} 
\end{proof}

\section{The proof of theorem \ref{main-theorem}}
Suppose $\kappa <\lambda$ are supercompact cardinals.
Consider the structure $V_\lambda=(V_\lambda,\in,\kappa, U)$, where $U$ is a suitable bookkeeping function  enumerating all $\mathbb Q^\kappa_\lambda$-terms in $V_\lambda$ which are forced to be  trees of size and height $\omega_1$ without cofinal branches.
Let $G_\lambda\subseteq\mathbb Q^\kappa_\lambda$  be a $V$-generic filter.
Since $\mathbb Q^\kappa_\lambda$ is $\lambda$-c.c,  every tree of size and height $\omega_1$ that has no cofinal branches is special in $V[G_\lambda]$. Recall that by Baumgartner's \cite[Theorem 7.5]{Baumgartnersurvey}, if every tree of height and size $\omega_1$ that has no  branches
of length $\omega_1$ is special, then every tree of height and size $\omega_1$ that has at most $\omega_1$ cofinal branches is weakly special.
Thus, by \cref{tree-guessing}, to show that ${\rm SGM}^+(\omega_3,\omega_1)$ holds in $V[G_\lambda]$, it suffices to show that
${\rm GM }^+(\omega_3,\omega_1)$ holds in
$V[G_\lambda]$. Indeed, we show that the $\omega_1$-guessing models of size $\omega_1$ witnessing 
${\rm GM }^+(\omega_3,\omega_1)$ are internally club.
\begin{lemma}\label{Magidor-guessing}
Let $\alpha \in E$. Suppose that $N\in \M_{G_\lambda}$ is a Magidor model with $\alpha \in N$.
Then $N[G_\alpha]$  is an $\omega_1$-guessing model in $V[G_\lambda]$. 
\end{lemma}
	
\begin{proof} 
Note that  the projection  $N \mapsto N\rest  \alpha$ is an isomorphism 
 and is the identity on  $\mathbb Q^\kappa_\alpha\cap N$.  Thus, $N[G_\alpha]$ and $(N \rest  \alpha) [G_\alpha]$ are also isomorphic. Therefore, by replacing $N$ with $N\rest  \alpha$ we may assume that it is an $\alpha$-model.  Let $\overline{N}$ be the transitive collapse of $N$, and let $\pi$ be the collapse map.  For convenience, let us write $\bar\kappa$ for $\kappa_N$.  Then $\overline{N}=V_{\bar\gamma}$, for some  $\bar\gamma$ with $\cof(\bar\gamma)\geq \bar\kappa$  and $\pi(\kappa)= \bar\kappa$. Let $\bar\alpha= \pi(\alpha)$.  Since $\alpha \in N$ and $N$ is an $\alpha$-model we have 
\[
\mathbb Q_N= \mathbb Q^\kappa_\lambda \cap N =
\mathbb Q^\kappa_\alpha \cap N.\]
Let   $\mathbb Q^{\bar\kappa}_{\bar\alpha}= \pi [\mathbb Q^\kappa_\alpha \cap N]$, and let $p\in G_\alpha$ be such that $N\in \M_p$. By \cref{U-strong} $p$ is $(N,\mathbb Q^\kappa_\alpha)$-strongly generic. 
Consequently,  $G_{\alpha,N}= G_\alpha \cap N$ is $V$-generic over $\mathbb Q^\kappa_\alpha \cap N$. 
It follows that  \[
G^{\bar\kappa}_{\bar\alpha}= \pi [G_{\alpha,N}]\]
is $V$-generic over $\mathbb Q^{\bar\kappa}_{\bar\alpha}$.
Note that  
\[
\overline{N[G_{\alpha,N}]}=
V_{\bar\gamma}[G^{\bar\kappa}_{\bar\alpha}]= 
V_{\bar\gamma}^{V[G_{\alpha,N}]},\]
which belongs to $V[G_{\alpha,N}]$.
It is clear that the pair $(\overline{N[G_{\alpha,N}]},V[G_{\alpha,N}])$ has the $\omega_1$-approximation property. 
On the other hand, by \cref{Key lemma in app},  the quotient forcing $\mathbb R^N_\alpha $ has the $\omega_1$-approximation property, and thus, by \cref{quotient-approximation-5.35}, the pair $(V[G_{\alpha,N}], V[G_\lambda])$ has the $\omega_1$-approximation property.
By \cref{transitivity of approx}, the pair
\[
(\overline{N[G_{\alpha,N}]},V[G_\lambda])
\]
has the $\omega_1$-approximation property. Now \cref{guessing vs approx} implies that $N[G_{\alpha,N}]$ is an $\omega_1$-guessing model in $V[G_\lambda]$. 
\end{proof}

One can show with a similar proof that if $\mu > \lambda$ and $N\prec V_\mu$ is a $\kappa$-Magidor model containing all the relevant parameters. Then $N[G_\lambda]$ is an $\omega_1$-guessing model in $V[G_\lambda]$.  We now prove that for all $\mu >\lambda$ the set of strongly $\omega_1$-guessing  models is stationary in ${\mathcal P}_{\omega_3}(V_\mu[G_\lambda])$. 
	
\begin{lemma}\label{lambda-Magidor-guessing}
Suppose that $\mu > \lambda$ and $N\prec V_\mu$ is a $\lambda$-Magidor model
containing all the relevant parameters. Then $N[G_\lambda]$ is a strongly $\omega_1$-guessing  model. 
\end{lemma}
	
\begin{proof} Since $N$ is a $\lambda$-Magidor model, its transitive collapse $\overline{N}$  equals $V_{\bar\gamma}$, for some $\bar\gamma < \lambda$. Let $\bar\lambda = N\cap \lambda$, which is in $E$.  Note that $\cof(\bar\lambda)\geq \kappa$, and hence the transitive collapse $\overline{N[G_\lambda]}$  of $N[G_\lambda]$ equals $V_{\bar\gamma}[G_{\bar\lambda}]$. On the other hand,  the pair $(V[G_{\bar\lambda}],V[G_\lambda])$ has the $\omega_1$-approximation property by \cref{level by level approx}.  Therefore, $V_{\bar\gamma}[G_{\bar\lambda}]$ and  hence also $N[G_\lambda]$ remains an $\omega_1$-guessing model in $V[G_\lambda]$. To see that $V_{\bar\gamma}[G_{\bar\lambda}]$   is a strongly $\omega_1$-guessing model,  fix some $\delta \in E$ with $\delta > \bar\gamma$ and ${\rm cof}(\delta)<\kappa$.  Note that if $M\in \M^\delta_{G_\lambda}$ is a Magidor model with $\bar\lambda \in M$, then by \cref{Magidor-guessing} $M[G_{\bar\lambda}]$ is an $\omega_1$-guessing model. Moreover, if $M\in \M_{G_\lambda}^\delta$ is a limit of such Magidor models then by \cref{continuity-Magidor},
\[
M\cap V_{\delta} = \bigcup \{ Q\cap V_\delta : Q\in_\delta M \mbox{ and } Q\in \M_{G_\lambda}^\delta\}.
\]
Hence if we let $\mathcal G$ be the collection of the models $(M\cap V_{\bar\gamma})[G_{\bar\lambda}]$, for Magidor models $M\in \M_{G_\lambda}^\delta$ with $\bar\lambda \in M$, then $\mathcal G$ is an $\in$-increasing chain of length $\omega_2$ which is continuous at $\omega_1$-limits and whose union is $V_{\bar\gamma}[G_{\bar\lambda}]$. Note that every model of size $\omega_1$ in this chain is a continuous union of an I.C. sequence. Therefore  $N[G_\lambda]$ is a  strongly $\omega_1$-guessing model in $V[G_\lambda]$, as required. 
\end{proof}

The following corollary is immediate and concludes the proof of \cref{main-theorem}.
\begin{corollary}
The principle ${\rm SGM}^+(\omega_3,\omega_1)$ holds in $V[G_\lambda]$. 
\end{corollary} 
\begin{proof}[\nopunct]
\end{proof}

\begin{remark}
The decoration component in the forcing conditions plays no role in our  main theorem. However, without additional effort, one can easily show that the $\omega_1$-guessing models of size $\omega_1$ witnessing the truth of ${\rm SGM}^+(\omega_3,\omega_1)$ in the final model are internally club.   
\end{remark}

\subsubsection*{Acknowledgements} The authors are grateful to the reviewer for his/her comments and suggestions.
		
\bigskip
\bibliographystyle{abbrv}
\bibliography{output}
\end{document}